\documentclass[journal]{IEEEtran}
\usepackage{cite}
\usepackage{xr}
\usepackage{xcite}
\usepackage{amsmath,amssymb,amsfonts}
\usepackage{multirow}
\usepackage{graphicx}
\usepackage{textcomp}
\usepackage{xcolor}
\usepackage{hyperref}
\usepackage{amsbsy}
\usepackage{amsthm, enumitem}
\usepackage{algorithm,algpseudocode}
\usepackage{graphicx}
\usepackage{epstopdf}
\usepackage[caption=false,font=footnotesize]{subfig}
\usepackage{booktabs} 
\usepackage{siunitx}  
\usepackage{xcolor} 

\algnewcommand{\Inputs}[1]{%
	\State \textbf{Inputs:}
	\Statex \hspace*{\algorithmicindent}\parbox[t]{.8\linewidth}{\raggedright #1}
}
\algnewcommand{\Initialize}[1]{%
	\State \textbf{Initialize:}
	\Statex \hspace*{\algorithmicindent}\parbox[t]{.8\linewidth}{\raggedright #1}
}

\makeatletter
\AtEndDocument{%
	\begingroup
	\edef\@currentlabel{\number\value{equation}}\label{ctr:last-equation}%
	\edef\@currentlabel{\number\value{lemma}}\label{ctr:last-lemma}%
	\endgroup
}
\makeatother

\usepackage{bm}

\newcommand{\br}[1]{\left({#1}\right)}
\newcommand{\bc}[1]{\left\{{#1}\right\}}
\newcommand{\bsq}[1]{\left[{#1}\right]}
\renewcommand{\b}[1]{\ensuremath{\mathbf{#1}}} 
\newcommand{\bs}[1]{\ensuremath{\boldsymbol{#1}}} 
\newcommand{\E}[1]{\ensuremath{\mathbb{E}\left[#1\right]}}  
\newcommand{\Et}[1]{\ensuremath{\mathbb{E}_t[#1]}}
\newcommand{\Ewt}[1]{\ensuremath{\mathbb{E}_{w_t}[#1]}}  
\newcommand{\norm}[1]{\ensuremath{\left\|#1\right\|}} 
\newcommand{\eqtext}[1]{\ensuremath{\stackrel{\text{#1}}{=}}} 
\newcommand{\leqtext}[1]{\ensuremath{\stackrel{\text{#1}}{\leq}}} 
\newcommand{\geqtext}[1]{\ensuremath{\stackrel{\text{#1}}{\geq}}} 
\DeclareMathOperator*{\argmin}{arg\,min} 
\providecommand{\ip}[2]{\langle #1, #2 \rangle} 

\renewcommand{\O}[1]{{\mathcal{O}\left(#1\right)}}

\newcommand{\lfrom}[1]{{\leqtext{\eqref{#1}}}}
\newcommand{\gfrom}[1]{{\geqtext{\eqref{#1}}}}
\newcommand{\eqfrom}[1]{{\eqtext{\eqref{#1}}}}

\def \x {{\b{x}}}
\def \X {{\b{X}}}
\def \y {{\b{y}}}
\def \z {{\b{z}}}

\def \v {{\b{v}}}

\def \u {{\b{u}}}

\def \p {{\b{p}}}

\def \I {{\mathbf{I}}}

\def \vxi {{\bs{\xi}}}
\def \th {{\bs{\theta}}}
\def \lmda {{\bs{\lambda}}}
\def \EE {{\mathbb{E}}}

\def \cX {{\mathcal{X}}}
\def \cO {{\mathcal{O}}}
\def \cOt {{\tilde{\mathcal{O}}}}
\def \cL {{\mathcal{L}}}
\def \cR {{\mathcal{R}}}

\def \cK {{\mathcal{K}}}

\def \cN {{\mathcal{N}}}
\def \cD {{\mathcal{D}}}
\def \tD {{\tilde{\Delta}}}
\def \cC {{\mathcal{C}}}

\def \lam {{\boldsymbol{\lambda}}}
\def \mub {{\boldsymbol{\mu}}}

\def \eps {{\epsilon}}

\def \Rn {{\mathbb{R}}}

\def \tx {{\tilde{\x}}}

\def \hx {{\hat{\x}}}
\def \nt {{\tilde{\nabla}}}
\def \bt {{\tilde{B}}}

\def \W {{\mathbf{W}}}

\def \T {{\mathsf{T}}}

\def \zz {{\mathbf{0}}}

\newtheorem{assumption}{}

\theoremstyle{remark}
\newtheorem{rem}{\bf Remark}
\newtheorem{theorem}{Theorem}
\newtheorem{lemma}{Lemma}

\begin{document}
	
	\title{Stochastic Sequential Quadratic Programming for Optimization with Functional Constraints}
	\author{
		Panchajanya Sanyal$^{*}$, Srujan Teja Thomdapu$^{*}$, and Ketan Rajawat\\
		\vspace{0.4cm}
		Department of Electrical Engineering\\
		Indian Institute of Technology Kanpur\\
		Kanpur, India
		\thanks{$^{*}$These authors contributed equally.}
	}
	\maketitle

	\begin{abstract}
		Stochastic convex optimization problems with nonlinear functional constraints are ubiquitous in signal processing applications including constrained least-squares, set-membership adaptive filtering, and trajectory optimization under uncertain fields. The presence of non-linear functional constraints renders the traditional projected stochastic gradient descent and related projection-based methods inefficient, and motivates the use of first-order methods. However, existing first-order methods, including primal and primal--dual algorithms, typically rely on a bounded (sub-)gradient assumption, which may be too restrictive in high-dimensional settings. We propose a stochastic sequential quadratic programming (SSQP) algorithm that works entirely in the primal domain, avoids projecting onto the feasible region, obviates the need for bounded gradients, and achieves state-of-the-art oracle complexity under standard smoothness and convexity assumptions. A faster version, namely SSQP-Skip, is also proposed where the quadratic subproblems can be skipped in most iterations. Finally, we develop an accelerated variance-reduced version of SSQP (VARAS), whose oracle complexity bounds match those for solving unconstrained finite-sum convex optimization problems. The superior performance of the proposed algorithms is demonstrated via numerical experiments on real datasets. 
	\end{abstract}
	
	
	\section{Introduction}
	We consider the constrained optimization problem
	\begin{align}\tag{$\mathcal{P}$}\label{mainProb}
		\begin{aligned}
			\x_\star = \arg\min_{\x \in \Rn^d} ~ &f(\x) + h(\x), \\
			\text{s. t. } \hspace{5mm} &g_k(\x) \leq 0, \hspace{1cm} 1\leq k\leq m
		\end{aligned}
	\end{align}
	where $f(\x) := \Et{f_{i_t}(\x)}$ and $\Et{\cdot}$ denotes the expectation with respect to the random index $i_t$. We will also consider a finite-sum case, which arises when the index $i_t$ is sampled uniformly from $\{1, \ldots, n\}$, so that $f(\x):=\frac{1}{n}\sum_{i=1}^n f_i(\x)$. The functions $f_i:\Rn^d \rightarrow \Rn$ and $g_k:\Rn^d \rightarrow \Rn$ are proper, closed, convex, and $L$-smooth. The regularization function $h:\Rn^d \rightarrow \Rn$ is convex but possibly non-smooth, and may include an indicator function corresponding to a set-inclusive constraint of the form $\x \in \cK$ for a closed convex set $\cK$. The stochastic objective function in \eqref{mainProb} commonly arises in the context of stochastic approximation  and online signal processing
		\cite{dieuleveut2023stochastic}, \cite{ribeiro2010ergodic}. The finite-sum structure in particular is widely used in batch estimation, system identification, and empirical risk formulations built from a dataset with $n$ training samples \cite{xin2020variance}. Non-linear functional constraints similarly arise in many signal processing tasks. For example, in trajectory generation and navigation under uncertain environmental fields \cite{yooEnsemble}, constraints encode vehicle kinematics and actuator limits. Similarly, in set membership adaptive filtering, constraints can enforce residual bounds for critical samples \cite{gollamudi1998set, bhotto2011robust, flores2019set}. 
	
	
	We consider the high-dimensional setting, where for a given $\x$, a stochastic first-order oracle (SFO) returns a stochastic gradient $\nabla f_{i_t}(\x)$, e.g., from a randomly selected data sample or snapshot, together with constraint values and gradients $\{g_k(\x), \nabla g_k(\x)\}_{k=1}^m$. In such settings, classical methods such as the projected stochastic gradient descent (SGD) lose their efficacy, since each iteration requires projection onto a feasible region defined by functional constraints, an operation that is  computationally expensive in high-dimensional signal processing models, and often impractical when constraints encode dynamics or per-sample performance requirements.  
	Instead, efficient and scalable algorithms for solving \eqref{mainProb} must rely only on the first-order information provided by the oracle. The \emph{SFO complexity} of an algorithm is defined as the number of SFO calls required to achieve an $\eps$-optimal solution, which may be a random vector $\x \in \Rn^d$ such that 
	\begin{align}\label{epsopt}
			&\E{f(\x) + h(\x)} - f(\x_\star) - h(\x_\star) \leq \eps, \\
			&\E{\max_k\{[g_k(\bar{\x}_T)]_+\}} \leq \eps.
	\end{align}
	For strongly convex objectives, we will directly characterize the complexity as the number of SFO calls required to ensure that $\EE\norm{\x-\x_\star}^2 \leq \epsilon$ for a random $\x$.

	State-of-the-art first-order algorithms for solving \eqref{mainProb} and its variants, like constrained online convex optimization (COCO), include primal algorithms \cite{necoara2022stochastic,lan2020algorithms,nedic2019random,basu2019optimal, bayandina2018mirror, stonyakin2019adaptive, alkousa2020modification, alkousa2019some} as well as primal--dual algorithms \cite{nesterov2009primal,nemirovski2009robust, xu2020primal,yazdandoost2019randomized,madavan2021stochastic,yan2022adaptive,yuan2018online,yu2017online,lin2018level}. Primal algorithms can be classified into two main categories: (a) stochastic versions of Polyak's subgradient method \cite{lan2020algorithms, basu2019optimal, bayandina2018mirror} that switch between either minimizing the objective function or reducing the infeasibility at every iteration; and (b) composite approaches  \cite{necoara2022stochastic, nedic2019random}, that involve both a proximal gradient step to reduce the objective and  a sub-gradient step to reduce the constraint violation at each iteration. On the other hand, the primal--dual methods in \cite{nesterov2009primal,nemirovski2009robust, xu2020primal,yazdandoost2019randomized,madavan2021stochastic,yan2022adaptive,yuan2018online,yu2017online,lin2018level} seem to follow a common template of updating both a primal and a dual variable at every iteration. When applied to solve \eqref{mainProb}, these approaches achieve an SFO complexity of $\cO\br{1/\eps^2}$ for the convex case and $\cO\br{1/\eps}$ for the strongly convex case. Additionally, all existing results require the boundedness of the (sub-)gradients of the objective function, a condition which may fail for common objectives such as the least-squares loss function or may be difficult to verify in practice. To the best of our knowledge, no existing method attains optimal SFO complexity for solving \eqref{mainProb} while simultaneously avoiding both projection steps and bounded (sub-)gradient assumptions. 
	
	This work puts forth a new class of stochastic sequential quadratic programming (SSQP) algorithms for solving \eqref{mainProb}.  The key idea behind the proposed approach is to reformulate \eqref{mainProb} as an unconstrained non-smooth optimization problem using the exact penalty method \cite{han1979exact,lin2003some} and to solve it 
	via the prox-linear algorithm \cite{zhang2021stochastic} using only stochastic first-order information. In the present case, each iteration reduces to solving a diagonal quadratic program (QP), instead of a full projection or general convex subproblems, while still achieving (near-)optimal SFO complexity and avoiding any bounded (sub-)gradient assumptions. Our main algorithmic contributions are: (i) a vanilla SSQP algorithm with SFO complexity matching state-of-the-art primal–dual algorithms for solving \eqref{mainProb}; (ii) an SSQP-Skip algorithm that solves the proximal and QP subproblems only infrequently, while retaining the same SFO complexity guarantees; and (iii) a first accelerated variance-reduced algorithm, VARAS, for finite-sum constrained optimization problems, whose SFO complexity is nearly on par with the best known methods in the unconstrained setting. Central to these algorithms is a novel one-step inequality that leads to the required optimality gap and constraint violation bounds. The numerical performance of the proposed variants is also tested on two standard signal processing problems and substantially outperforms representative baselines for solving \eqref{mainProb}. The proposed algorithms differ fundamentally from classical sequential QP  methods, which require full gradient and Hessian information at each iteration \cite[Sec. 4.3.1]{bertsekas} and are therefore ill-suited to large-scale stochastic settings; in contrast, the proposed SSQP methods operate with only stochastic first-order information. Beyond the stochastic or finite-sum settings, the proposed framework is also compatible with real-time or streaming regimes, where its tracking performance can be studied using similar tools \cite{bedi2018tracking}.
	
	\subsection{Related Work}
	Sequential quadratic programming (SQP) methods are among the most effective approaches for solving nonlinear optimization problems of the form \eqref{mainProb} \cite{gill2012sequential}. Conventional SQP methods rely on second-order derivatives of $f_i$ and $\{g_k\}_{k=1}^m$ to solve a sequence of QP problems subject to linearized constraints \cite{han1979exact}, and have been widely applied to mixed-integer nonlinear programming and nonlinear optimization with nonlinear equality constraints; see e.g. \cite{gill2012sequential, curtis2024sequential}.
	
	The exact penalty reformulation of constrained problems has been well studied in convex optimization, and a sequential QP approach is detailed in \cite[Sec.~4.3.1]{bertsekas}. The objective of the exact-penalty reformulation can be viewed as a compositional optimization problem, and in the deterministic case it can be solved using the approach of \cite{doikov2022optimization}. The corresponding stochastic problems can be addressed using the model-based framework of \cite{davis2019stochastic} or the prox-linear method of \cite{zhang2021stochastic}. A different exact-penalty method was proposed in \cite{wang2017penalty}, where the focus is on non-convex equality constraints, leading to an oracle complexity of $\mathcal{O}(\eps^{-3.5})$ for finding an $\eps$-stationary point. Using smooth approximations of exact penalties, so as to work with their gradients, is another alternative but typically yields weaker rates \cite{xiao2019penalized, thomdapu2019optimal}. Our work can be viewed as bringing these exact-penalty and prox-linear ideas into a stochastic SQP framework with explicit oracle-complexity guarantees.
	
	Convex optimization problems with functional constraints have been extensively studied; see \cite{bertsekas2014constrained} and references therein. When the constraint sets are difficult to project onto, projected SGD variants can be computationally expensive. The number of projections has been reduced to $\cO(\log T)$ in \cite{zhang2013logt, chen2016optimal} and to a single projection in \cite{mahdavi2012stochastic}, though these schemes can still be impractical when $m$ is large. Subsequent works that completely avoid projections include primal--dual methods \cite{xu2020primal, yazdandoost2019randomized, madavan2021stochastic, yan2022adaptive, yuan2018online, yu2017online, lin2018level} and primal methods \cite{necoara2022stochastic, lan2020algorithms, nedic2019random, basu2019optimal, bayandina2018mirror}. To the best of our knowledge, these projection-free methods have not been accelerated to attain the optimal rates that are possible in the unconstrained setting. The proposed work can be viewed as lying between these two classes: it avoids projections, instead solving a diagonal QP with linear constraints, while achieving (near-)optimal rates at par with accelerated variance-reduced projected-SGD \cite{lan2019unified}. In addition, and unlike many of the above approaches, our analysis does not require boundedness of the gradients of the constituent functions.

	Convex optimization problems with linear inclusive constraints have been studied in \cite{jalilzadeh2021primal, fercoq2019almost, kundu2018convex}; in contrast, we focus here on nonlinear functional constraints. Our formulation also differs from those in \cite{madavan2021stochastic, yan2022adaptive, yu2017online, lin2018level, lan2020algorithms, basu2019optimal, boob2023stochastic, akhtar2021conservative}, where the constraints are only required to hold on average, i.e., $\E{f_i(\x,\xi)} \leq 0$. Related stochastic formulations with more general stochastic function classes are considered in \cite{thomdapu2019optimal, thomdapu2021optimizing, thomdapu2023stochastic}. Other works address problems with infinitely many functional constraints; see \cite{xu2020primal, yazdandoost2019randomized, necoara2022stochastic, necoara2021minibatch, nedic2019random, wang2016stochastic}. The best convergence rates obtained in these papers are of order $\cO(1/\eps^2)$ for convex objectives and $\cO(1/\eps)$ for strongly convex objectives. These rates have only been improved in \cite{wei2018solving, yang2017richer} for a specific formulation imposing additional structure on the constraint functions. Finally, \cite{guo2022online, sinha2024optimal} study related problems from the perspective of constrained online convex optimization (COCO), where the lack of stationarity assumptions leads to more conservative bounds. In contrast, in this paper we propose a new method for problems of the form \eqref{mainProb} and establish (near-)optimal oracle-complexity rates for both convex and strongly convex objectives under our setting.
	
	We remark that among these, \cite{xu2020primal, necoara2022stochastic, necoara2021minibatch, nedic2019random} adopt a different SFO model for \eqref{mainProb}, wherein at iteration $t$, only $\{g_{j_t}(\x), \nabla g_{j_t}(\x)\}$ for random index $j_t$, is revealed. While this greatly reduces the per-iteration cost, these algorithms require strong regularity assumptions that couple individual constraint violations with the distance to the full feasible set (see, e.g., \cite[Assumption 4]{necoara2022stochastic}), as well as bounded subgradients of both the objective and the constraints. Their resulting convergence rates match those of stochastic subgradient methods and depend explicitly on the regularity constant of the constraint system. In contrast, the proposed SSQP method uses all functional constraint gradients at every iteration, and therefore does not rely on such regularity conditions, while attaining (near-)optimal oracle complexity. Nevertheless, a single-constraint variant of SSQP with improved rates may be an interesting direction for future work.	
	
	Different from these, the random constraint projection method in \cite{wang2016stochastic2} projects onto the sampled constraint set $\{\x \mid g_{j_t}(\x) \leq 0\}$ at iteration $t$. Although \cite{wang2016stochastic2} relaxes the bounded-subgradient assumption, its guarantees are not directly comparable to ours. In particular, boundedness of the iterates is assumed, enforced through an additional auxiliary projection, and the bounds are established for projected averages rather than for the returned iterate.
	
	A comparison of the proposed methods with the most relevant state-of-the-art algorithms is summarized in Table~\ref{constr-papers}. The table excludes works such as \cite{curtis2024sequential}, which provide only asymptotic convergence guarantees, but includes bounds for online algorithms adapted to the present setting. We also omit the decentralized non-convex extension of our approach presented in \cite{idrees2025decentralized}. 
	
	We remark that the table does not list the per-iteration computational costs, which could depend on the implementation and problem structure. While several primal-dual and online methods have $\cO(m)$ constraint-related costs per iteration, the computational complexity of the other algorithms depends on the implementation and problem structure. For instance, random constraint projection methods treat the projections as oracle operations \cite{wang2016stochastic2}. In practice, however, projecting onto \(\{\x \mid g_k(\x) \leq 0\}\) may itself be non-trivial and require a custom iterative solver. The level-set method in \cite{lin2018level} requires solving a nonsmooth level-set subproblem at each outer iteration, whose computational complexity depends on the problem structure. Likewise, the final projection required in \cite{mahdavi2012stochastic} may be difficult for nonlinear functional constraints. In contrast, the proposed algorithms require solving QP subproblems for which mature off-the-shelf solvers are readily available, although the worst-case costs  can be as high as $\cO(m^3)$.
	
	\begin{table}[ht]
		\centering
		\caption{Related works solving \eqref{mainProb} with state-of-the-art complexity. Here, SCP stands for sequential convex programming and SQP for sequential quadratic programming. }
		\label{constr-papers}
		\begin{tabular}{|c|c|c| c|}
			\hline
			\multirow{2}{*}{Ref} & \multirow{2}{*}{Method class} & \multicolumn{2}{c|}{SFO Complexity} \\ \cline{3-4}
			& & \multicolumn{1}{c|}{Convex} & Strongly convex \\ \hline
			\cite{mahdavi2012stochastic} & primal--dual & \multicolumn{1}{c|}{$\cO\br{\frac{1}{\eps^2}}$} & $\cOt\br{\frac{1}{\eps}}$ \\ \hline
			\cite{yan2022adaptive, madavan2021stochastic} & primal--dual & \multicolumn{1}{c|}{$\cO\br{\frac{1}{\eps^2}}$} & - \\ \hline
			\cite{yuan2018online} & primal--dual (online) & $\cO\br{\frac{1}{\eps^2}}$ & $\cOt\br{\frac{1}{\eps}}$ \\ \hline
			\cite{lin2018level} & primal & \multicolumn{1}{c|}{$\cO\br{\frac{1}{\eps^2}}$} & - \\ \hline
			\cite{lan2020algorithms} & primal & \multicolumn{1}{c|}{$\cO\br{\frac{1}{\eps^2}}$} & $\cO\br{\frac{1}{\eps}}$ \\ \hline
			\multirow{2}{*}{\cite{necoara2022stochastic}} & primal & $\cO\br{\frac{1}{\eps^2}}$ & $\cO\br{\frac{1}{\eps}}$ \\ \hline
			\cite{bayandina2018mirror} & primal & \multicolumn{1}{c|}{$\cO\br{\frac{1}{\eps^2}}$} & $\cO\br{\frac{1}{\eps}}$ \\ \hline
			\cite{wang2016stochastic2} & random projection & \multicolumn{1}{c|}{$\cO\br{\frac{1}{\eps^2}}$} & - \\ \hline
			\multirow{2}{*}{\cite{guo2022online}} & SCP (online) & $\cO\br{\frac{1}{\eps^2}}$ & $\cOt\br{\frac{1}{\eps}}$ \\ \hline
			\begin{tabular}[c]{@{}c@{}}SSQP\\(Ours)\end{tabular} & SQP & \multicolumn{1}{c|}{$\cO\br{\frac{1}{\eps^2}}$} & $\cO\br{\frac{1}{\eps}}$ \\ \hline
			\begin{tabular}[c]{@{}c@{}}SSQP-Skip\\(Ours)\end{tabular} & SQP & \multicolumn{1}{c|}{-} & $\cO\br{\frac{1}{\eps}}$ \\ \hline
			\begin{tabular}[c]{@{}c@{}}VARAS\\(Ours)\end{tabular} & SQP & \multicolumn{1}{c|}{\scriptsize $\cO\br{\sqrt{\frac{n}{\epsilon}}}$} & \multicolumn{1}{c|}{\tiny $\cO\br{n \log n + \sqrt{n}\log\tfrac{1}{\eps}}$} \\ \hline
		\end{tabular}
	\end{table}
	
	\subsection{Organization}
	This paper is organized as follows. Sec. \ref{sec-prelim} provides some preliminaries, including the assumptions, background, and basic inequalities. Sec. \ref{sec-ssqp} details the proposed SSQP and SSQP-Skip algorithms for solving the general stochastic version of \eqref{mainProb}, and provides their oracle complexity bounds. Sec. \ref{sec-avrssqp} develops the accelerated variance-reduced SSQP algorithm for the finite-sum version of \eqref{mainProb} and provides the corresponding oracle complexity bounds. The numerical performance of the proposed class of algorithms is provided in Sec. \ref{sec-numerical} and finally, Sec. \ref{sec-conclusion} concludes the paper. 
	
	\subsection{Notation}
	We use regular (bold)-faced letters to represent scalars (column vectors). We let $[v]_+:= \max\{0,v\}$, so that $\max\bc{[v_k]_+}= \max\bc{[v_1]_+, \ldots ,[v_m]_+} = \max\bc{0, v_1, \ldots, v_m}$. The Euclidean norm of a vector $\x$ is denoted by $\norm{\x}$ and $\E{.}$ denotes the expectation operator.
	
	\section{Preliminaries}\label{sec-prelim}
	This section contains the assumptions and some technical claims that we use throughout the analysis. Other basic mathematical inequalities that are used throughout the text are listed in Appendix \ref{basic}. 
	
	\subsection{Assumptions}
	\begin{assumption}\label{slater}
		The Slater condition holds for \eqref{mainProb}, i.e., there exists a feasible $\tx$ such that 
		\begin{align}
			g_k(\tx) &\leq -\nu < 0, & 1\leq k \leq m.	
		\end{align}
		Additionally, we assume that the optimality gap at the Slater point is bounded, i.e., $f(\tx) + h(\tx) - f(\x_\star) - h(\x_\star) \leq \bt$.
	\end{assumption}
	In the context of constrained optimization, the Slater constraint qualification (CQ) is one of the classical CQs that imply strong duality and existence of a primal--dual optimum pair $(\x_\star,\lam_\star)$. Since \eqref{mainProb} is convex, the optimum pair satisfies the Karush-Kuhn-Tucker (KKT) conditions, so that
	\begin{align}\label{optimalKKT}
		f(\x_\star)+h(\x_\star) &= \min_{\x  \in \Rn^d} f(\x) + h(\x) + \sum_{k=1}^m \lambda_{k,\star} g_k(\x) \\
		&\leq f(\tx) + h(\tx) - \norm{\lam_\star}_1 \nu\text{ ,}
	\end{align} 
	for the Slater point $\tx$. Rearranging and using Assumption \ref{slater}, we obtain the bound $\norm{\lam_\star}_1 \leq \frac{\bt}{\nu}$. In practice, $\nu$ and $\bt$ are problem parameters that must be found by parameter tuning. For example, suppose that we have found a Slater point $\tx$ so that $\nu = -\max_k g_k(\tx)$. Then if we can also find the unconstrained minimum $\x_u = \arg\min_{\x \in \Rn^d} (f(\x) + h(\x))$, we can set $\bt = f(\tx) + h(\tx) - f(\x_u) - h(\x_u)$. 
	
	We remark that Assumption \ref{slater} can be replaced by a weaker, but harder to verify, requirement that a KKT point $(\x_\star, \lam_\star)$ exists, as also done in \cite{xu2020primal}. All subsequent complexity bounds continue to hold under this weaker condition as well, with every occurrence of  $\bt/\nu$ replaced by $\norm{\lam_\star}_1$. In comparison, Slater's CQ is often easy to verify, but is a stronger sufficient condition and may fail for some equivalent formulations, as illustrated in \cite{xu2020primal}. An alternative regularity assumption often used in related works is the linear regularity of the constraint system \cite{necoara2022stochastic,nedic2019random,necoara2021minibatch, wang2016stochastic2}. Although Slater's CQ implies bounded linear regularity on compact sets through classical error-bound results \cite[Corollary~5]{bauschke1999strong}, it does not in general imply the global regularity assumptions used in these works, which are stated over the full decision region.
	
	The next few assumptions define the other problem parameters used for the analysis. The first set of assumptions is standard. 
	\begin{assumption}\label{a1}
		The following assumptions hold:
		\begin{enumerate}
			\item The functions $f_i$ and $g_k$ are proper, closed, and convex.
			\item The functions $f_i$ are $L_f$-smooth while the functions $g_k$ are $L_g$-smooth. 
			\item The function $f$ is $\mu$-strongly convex for $\mu \geq 0$. 
		\end{enumerate}   
	\end{assumption}
	The strong convexity assumption may not be invoked for some of the proposed algorithms, for which we will simply set $\mu = 0$. The following assumption will be required for the general stochastic optimization problem, but will be dropped for the finite-sum case. 
	\begin{assumption}\label{gradnoi}
		The gradient noise at the optimum $\x_\star$ is bounded as $\Et{\norm{\nabla f(\x_\star) - \nabla f_{i_t}(\x_\star)}^2} \leq \sigma^2$. 
	\end{assumption}
	The bounded gradient noise assumption is again standard in the literature and is always required for SGD-like algorithms in general. Note that we only require the gradient noise to be bounded at a specific point $\x_\star$ rather than for the entire domain $\cK$. Assumption \ref{gradnoi} along with the smoothness of $f_{i_t}$ implies that for a given $\x$, 
	\begin{align}
		&\Et{\norm{\nabla f_{i_t}(\x)-\nabla f(\x)}^2} \\
		&= \Et{\norm{\nabla f_{i_t}(\x)-\nabla f(\x_\star)}^2} - \Et{\norm{\nabla f(\x_\star)-\nabla f(\x)}^2}\nonumber\\
		&\lfrom{peterpaul} 2\Et{\norm{\nabla f_{i_t}(\x_\star)-\nabla f_{i_t}(\x)}^2}\nonumber \\
		&\hspace{1cm}+ 2\Et{\norm{\nabla f(\x_\star)-\nabla f_{i_t}(\x_\star)}^2}\nonumber\\
		&\lfrom{gradnoi} 2\Et{\norm{\nabla f_{i_t}(\x_\star)-\nabla f_{i_t}(\x)}^2} + 2\sigma^2 \label{sm0} \\
		&\lfrom{coco} 4L_f\Et{f_{i_t}(\x_\star) - f_{i_t}(\x) - \ip{\nabla f_{i_t}(\x)}{\x_\star - \x}} + 2\sigma^2 \nonumber\\
		&\leq 4L_fD_f(\x_\star,\x)+ 2\sigma^2, \label{gradvar1}
	\end{align}
	where $D_f(\u,\v) := f(\u) - f(\v) - \ip{\nabla f(\v)}{\u - \v}$ and we have used some basic inequalities from Appendix \ref{basic}. The inequality in \eqref{gradvar1} bounds the gradient noise at arbitrary $\x$ in terms of the Bregman divergence between $\x$ and $\x_\star$ (with respect to $f$) and the gradient noise at $\x_\star$, and will turn out to be useful later on. Finally, we have the following initialization condition. 
	
	\begin{assumption}\label{init}
		All algorithms can be initialized with arbitrary, possibly infeasible $\x_0 \in \cK$ which satisfies $\norm{\x_0-\x_\star} \leq B_x$ and 
		\begin{align*}
			f(\x_0) + h(\x_0) + \gamma\sum_{k=1}^m [g_k(\x_0)]_{+} - f(\x_\star) - h(\x_\star) \leq B_{\gamma},
		\end{align*}
		for a given $\gamma > 0$. 
	\end{assumption}
	
	\subsection{Exact Penalty Reformulation}\label{penaltyequivalent}
	We reformulate the problem using the exact penalty method \cite[Sec. 4.3.1]{bertsekas}, so as to obtain:
	\begin{align}
		\x_\star &= \arg\min_{\x\in \Rn^d} ~ F(\x) := f(\x) + h(\x) + \gamma\max\{[g_k(\x)]_+\} \nonumber \\
		&= \arg\min_{\x\in\Rn^d, v \geq 0} f(\x) + h(\x) + \gamma v \tag{$\mathcal{P}_1$}\label{penalty1}\\
		& \text{s. t. }	~ g_k(\x) \leq v, 	\hspace{1cm} 1\leq k \leq m \label{epi1}
	\end{align}
	where recall that $\max\{[v_k]_+\} = \max\{0,v_1,v_2, \ldots, v_m\}$. In general, the solution of \eqref{penalty1} is the same as that of \eqref{mainProb} for sufficiently large $\gamma$. Specifically, under Assumption \ref{slater}, it suffices to set $\gamma \geq \frac{\bt}{\nu}$. To see this, associate dual variables $\mu_k \geq 0$ with the $k$-th constraint in \eqref{penalty1}, so that the Lagrangian becomes:
	\begin{align}
		L(\x,v,\mub) &= f(\x) + h(\x) + \gamma v  + \sum_{k=1}^m \mu_k (g_k(\x) - v) \\
		&\hspace{-5mm}= f(\x) + h(\x) + \sum_{k=1}^m \mu_k g_k(\x) + v(\gamma - \norm{\mub}_1)
	\end{align} 
	where $\mub \in \Rn^{m}_+$ collects the dual variables $\{\mu_k\}_{k=1}^m$. Since the Slater CQ is satisfied by \eqref{mainProb}, it is also satisfied by \eqref{penalty1}. Therefore, the first order KKT point $(\x_\star,v_\star,\mub_\star)$ is such that
	\begin{align}
		(\x_\star,v_\star) = \arg\min_{\x \in \Rn^d, v\geq 0} &f(\x) + h(\x) + \sum_{k=1}^m \mu_{k,\star} g_k(\x) \\
		&+ v(\gamma - \norm{\mub_\star}_1)
	\end{align}
	Hence, for $\gamma = \frac{\bt}{\nu} \geq \norm{\mub_\star}_1$, it follows that $v_\star = 0$ and consequently $(\x_\star, \mub_\star)$ is KKT-optimal for \eqref{mainProb}. 
	
	The requirement $\gamma\ge\norm{\mub_\star}_1$ is analogous to other problem-dependent constants that appear in constrained stochastic optimization works, such as the linear regularity constant in \cite{necoara2022stochastic, mahdavi2012stochastic, wang2016stochastic2} or the diameter of compact decision sets in \cite{lan2020algorithms,yuan2018online, guo2022online}. In practice, the theoretically sufficient value of $\gamma$ may turn out to be too conservative and may require tuning.
	
	In addition to characterizing the SFO complexity, we observe that the SQP methods require solving a QP with $m$ linear constraints at every iteration.  Hence for this class of algorithms, we assume access to the quadratic minimization oracle (QMO) which can provide the solution to a given QP with $m$ linear constraints. In this case, for general convex objectives, we will characterize the performance of the algorithms in terms of the number of SFO and QMO calls required to achieve an $\epsilon$-optimal solution. 
	
	The exact penalty reformulation confirms that for any solution $\x_\star$ of \eqref{mainProb}, $\Delta_t := \E{F(\x_t)}-F(\x_\star)$ is a non-negative quantity. All subsequent theorems will involve an intermediate step of upper bounding $\Delta_t$ or related quantities, which will then lead to the required SFO and QMO complexity results for \eqref{epsopt}. 
	In particular, if an output point $\x$ satisfies $\E{\max_k [g_k(\x)]_+}\le \eps$, then Markov's inequality gives
		$\mathbb P\left(\max_{1\le k\le m} [g_k(\x)]_+>\kappa\right)\le \frac{\eps}{\kappa}$ for any $\kappa>0$. For instance, choosing $\kappa=\sqrt{\eps}$ yields that the probability of any $\sqrt{\eps}$-constraint violation is at most $\sqrt{\eps}$.

	\section{Stochastic Sequential Quadratic Programming Method} \label{sec-ssqp}
	In this section, we consider the general stochastic problem in \eqref{mainProb}. Reformulating the problem as \eqref{penalty1} makes it amenable to the application of stochastic proximal gradient methods. Specifically, we utilize the stochastic prox-linear algorithm from \cite{davis2019stochastic, zhang2021stochastic,duchi2018stochastic,drusvyatskiy2019efficiency} to develop the proposed SSQP algorithm. Throughout this section, we will focus on obtaining state-of-the-art rates but ignore constants or higher-order terms. Future work may target sharper constants and lower bounds under this oracle. 
	
	\subsection{SSQP Algorithm}
	The SSQP algorithm entails performing a partial linearization of the objective in \eqref{penalty1}, adding a proximal penalty, and minimizing the resulting quadratic form. Specifically, the objective and constraint functions are linearized, but the regularizer, as well as the $\max\{[\cdot]_+\}$ operator are not disturbed. Starting at an arbitrary $\x_0$, the updates of the proposed SSQP algorithm take the form:
	\begin{align}
		\x_{t+1} &= \argmin_{\u\in\Rn^d} \Bigl\{\ip{\nabla f_{i_t}(\x_t)}{\u} + h(\u) + \frac{1}{2\eta_t}\norm{\x_t - \u}^2 \nonumber\\
		& + \gamma \max\{[g_k(\x_t) + \ip{\nabla g_k(\x_t)}{\u-\x_t}]_+\} \label{proxlin-x}\Bigr\}\\
		&\hspace{-6mm}=\argmin_{\u \in \Rn^d,v\geq 0} \bigl\{\ip{\nabla f_{i_t}(\x_t)}{\u} + h(\u) + \frac{1}{2\eta_t}\norm{\x_t - \u}^2 + \gamma v\bigr\} \nonumber\\
		&\hspace{-2mm}\text{s. t. } g_k(\x_t) + \ip{\nabla g_k(\x_t)}{\u-\x_t} \leq v, ~ 1\leq k \leq m,\label{proxlin-x-epi}
	\end{align}
	for $t\geq 0$, where $i_t$ is a random index. Observe that when $\cK = \Rn^d$, the updates bear resemblance to the sequential quadratic programming approach proposed in \cite[Sec. 4.3.1]{bertsekas}, and hence we refer to our algorithm, summarized in Algorithm \ref{ssqp}, as Stochastic SQP (SSQP).
	
	\begin{algorithm}
		\caption {SSQP}
		\begin{algorithmic}[1]
			\State\textbf{Input:} $\x_0\in \cK$, $\gamma = \bt/\nu$, and $\eta_t \in (0,1]$.
			\State\textbf{for} $t =0,1,...,T-1$
			\State\hspace{3mm} Sample $i_t$ randomly
			\State\hspace{3mm} Update $\x_{t+1}$ using \eqref{proxlin-x}
			\State\textbf{end for}
			\State\textbf{Output:} $\bar{\x}_T = \frac{\sum_{t = 1}^T\eta_t\x_t}{\sum_{t=1}^T \eta_t}$.	
		\end{algorithmic}
		\label{ssqp}
	\end{algorithm}

	Since SSQP is a special case of the stochastic prox-linear algorithm, the $\cO\br{1/\sqrt{t}}$ convergence result for convex objectives follows from \cite{davis2019stochastic}. However, the generality of the prox-linear algorithm leads to relatively weaker bounds and requires stronger assumptions. Below, we provide a tighter bound which does not require the bounded gradients assumption commonly required for analyzing prox-linear algorithms. 
	
	The convergence of Algorithm \ref{ssqp} is established in the statement of the following theorem, whose proof is provided in Appendix \ref{pfprox-thm}.
	\begin{theorem}\label{prox-thm}
		Under Assumptions \ref{slater}-\ref{init}, $L = \max\{\gamma L_g,L_f\}$, and $\delta_0 =\norm{\x_0-\x_\star}^2$, we have the following SFO and QMO complexity bounds. 
		\begin{enumerate}[leftmargin=*, wide]
			\item For a convex objective, using the stepsize $\eta_t = \frac{\eta_0}{\sqrt{T}}$, where $\eta_0=\min\{\frac{\sqrt{\delta_0}}{2\sigma}, \frac{1}{4L}\}$, we obtain
			\begin{align}
				\E{f(\bar{\x}_T) + h(\bar{\x}_T)}-&f(\x_\star) - h(\x_\star) \nonumber\\
				&\leq \frac{2}{\sqrt{T}}\max\{2L\delta_0, \sigma\sqrt{\delta_0}\}, \label{thm1-conv}\\
				\E{\max_k\{[g_k(\bar{\x}_T)]_+\}} &\leq \frac{2}{(\gamma-\frac{\bt}{\nu})\sqrt{T}}\max\{2L\delta_0, \sigma\sqrt{\delta_0}\}, \label{thm1-constraint}
			\end{align}
			and an SFO/QMO complexity of $\O{\frac{\max\{\delta_0^2L^2, \sigma^2\delta_0\}}{\epsilon^2}}$. 
			\item For a $\mu$-strongly convex objective, the stepsize $\eta_t = \frac{2}{\mu(t+\lfloor16\kappa\rfloor+1)}$ with $\kappa = L/\mu$ results in the bound
			\begin{align}
				\E{\norm{\x_{T}-\x_\star}^2} &\leq \frac{8\sigma^2}{\mu^2T} + \frac{(16\kappa+2)^3\delta_0}{T^3}, 
			\end{align}
			and an SFO/QMO complexity of $\O{\frac{\sigma^2}{\mu^2\epsilon}+\frac{\kappa\delta_0^{1/3}}{\epsilon^{1/3}}}$. 
		\end{enumerate}
	\end{theorem}
	Proof of Theorem \ref{prox-thm} relies on an important one-step inequality (Lemma \ref{step}) that is reminiscent of but different from corresponding inequality in the proximal SGD setting.  As in proximal SGD, using a diminishing stepsize in the convex case yields a slightly worse bound of $\O{\frac{\log(T)}{\sqrt{T}}}$ in \eqref{thm1-conv}. Ignoring the constant terms, the SFO complexity results in Theorem \ref{prox-thm} also match the best known bounds for proximal SGD \cite{khaled2023unified, gorbunov2020unified} for both convex and strongly convex objectives. The obtained rates also match the best known rates achieved for the functional constrained problems in  \cite{mahdavi2012stochastic,madavan2021stochastic,yan2022adaptive,yu2017online,lin2018level,lan2020algorithms,basu2019optimal,bayandina2018mirror} while not requiring any bounded gradient assumptions. 
	
	Note that the smoothness of the objective function is critical for establishing the one-step inequality in Lemma~\ref{step}. A nonsmooth extension of Thm. \ref{prox-thm} would likely require bounded-subgradient or subgradient-growth assumptions on the stochastic objective components, which is the assumption regime already covered by existing methods such as \cite{necoara2022stochastic,nedic2019random}.
	
	The bounds and proof of Theorem \ref{prox-thm} reveal a deeper connection to proximal SGD, suggesting that recent advances in the proximal SGD literature can be leveraged to design even faster algorithms for \eqref{mainProb}. To demonstrate this idea in practice, we next present the algorithm that skips the step of solving QP in the intermittent iterative steps. In the next section, we will consider the finite-sum variant of \eqref{mainProb} and develop an accelerated and variance-reduced version of SSQP. The bounds in Thm. \ref{prox-thm} depends on $\gamma$ through $L=\max\{\gamma L_g,L_f\}$. Therefore, choosing an overly conservative penalty parameter can worsen the theoretical complexity. Adaptive selection of $\gamma$, for example by increasing it based on the dual value of the per-iteration subproblem as done in deterministic SQP methods, may improve practical performance. Developing such an adaptive stochastic variant with matching guarantees is left for future work.
	

	\subsection{SSQP-Skip algorithm}
	We now consider a situation when solving the QP with $m$ linear constraints is more expensive than evaluating  $\nabla f_{i_t}(\cdot)$ for a random $i_t \in \bc{1,...,n}$. This may be the case, for instance, when $m$ is large,  or if the proximal operator with respect to the regularizer $h$ is complicated, e.g., when $h$ is an indicator function corresponding to complicated set constraints. In such situations, it may be desirable to have an algorithm that allows one to skip solving the QP at most iterations. To this end, we put forth the SSQP-Skip algorithm which, for smooth and strongly convex functions, requires only $\O{1/\sqrt{\epsilon}}$ calls to the QMO, as opposed to the $\O{1/\epsilon}$ calls required by Algorithm \ref{ssqp} as per Theorem \ref{prox-thm}. 
	
	The proposed SSQP-Skip algorithm builds upon a similar SProxSkip algorithm from \cite{mishchenko2022proxskip}, but incorporates constraints and yields slightly better bounds. Specifically, we maintain an auxiliary variable $\y_t$ that is used in place of $\nabla f_{i_t}(\x_t)$ in \eqref{proxlin-x}, resulting in the updates:
	\begin{align}
		\hx_{t+1} &= \argmin_{\u\in\Rn^d} \Bigl\{\ip{\y_t}{\u} + h(\u) + \frac{p_t}{2\eta_t}\norm{\tx_{t+1} - \u}^2 \nonumber\\
		&\hspace{-5mm} + \gamma \max\{[g_k(\tx_{t+1}) + \ip{\nabla g_k(\tx_{t+1})}{\u-\tx_{t+1}}]_+\}\Bigr\} \label{ssqp-skip1}\\
		&\hspace{-6mm}=\argmin_{\u \in \Rn^d, v\geq 0} \bigl\{\ip{\y_t}{\u} + h(\u) + \frac{p_t}{2\eta_t}\norm{\tx_{t+1} - \u}^2 + \gamma v\bigr\} \nonumber\\
		&\hspace{-2mm}\text{s. t. } g_k(\tx_{t+1}) + \ip{\nabla g_k(\tx_{t+1})}{\u-\tx_{t+1}} \leq v, ~ 1\leq k \leq m,\nonumber
	\end{align}
	which are carried out with probability $p_t \ll 1$. Observe that compared to \eqref{proxlin-x}, the update in \eqref{ssqp-skip1} also utilizes a modified stepsize parameter $\eta_t/p_t$ and entails linearizing $g_k$ around $\tx_{t+1}$, which is given by 
	\begin{align}
		\tx_{t+1} &= \x_t - \eta_t\br{\nabla f_{i_t}(\x_t) - \y_t}. \label{sgd-up}
	\end{align}
	For the subsequent iteration, we update $\x_{t+1} = \hx_{t+1}$ when \eqref{ssqp-skip1} is evaluated (with probability $p_t$) and keep $\x_{t+1} = \tx_{t+1}$ otherwise (with probability $1-p_t$). Finally, the auxiliary variable $\y_{t+1}$ is kept the same when the QP is not solved, but updated whenever the QP is solved. The proposed approach is summarized in  Algorithm \ref{ssqp-skip}. Clearly, when $p_t$ is small, Algorithm \ref{ssqp-skip} needs to solve the QP in \eqref{ssqp-skip1} only rarely.

	\begin{algorithm}
		\caption {SSQP-Skip}
		\begin{algorithmic}[1]
			\State\textbf{Input:} $\x_0 \in \cK$, $\y_0 = \nabla f_{i_0}(\x_0)$, $\gamma = \bt/\nu$, $p >0$, and $\eta_t\in(0,1]$ and $i_0$ is a random index. 
			\State\textbf{for} $t = 0,1,...,T-1$
			\State\hspace{3mm} Sample $i_t$ randomly 
			\State\hspace{3mm} Evaluate $\tx_{t+1}$ as per \eqref{sgd-up}
			\State\hspace{3mm} Sample $w_t \sim $ Bernoulli$(p_t)$ and update
			\begin{align}
				\x_{t+1} &= w_t\hx_{t+1}  + (1-w_t)\tx_{t+1}, \label{ssqp-up}
			\end{align}
			\hspace{4mm}where $\hx_{t+1}$ is calculated as per \eqref{ssqp-skip1}.
			\State\hspace{3mm} Update $\y_{t+1} = \y_t + \frac{p_t}{2\eta_t}\br{\x_{t+1}-\tx_{t+1}}$
			\State\textbf{end for}
			\State\textbf{Output:} $\x_T$.	
		\end{algorithmic}
		\label{ssqp-skip}
	\end{algorithm} 
	
	Having detailed the proposed SSQP-Skip algorithm, the following theorem characterizes its performance.
	\begin{theorem}\label{ssqp-skip-thm}
		Under Assumptions \ref{slater}-\ref{init}, and $L = \max\{\gamma L_g,L_f\}$, $\kappa = L/\mu$, $\eta_t = \frac{2}{\mu(t+1+\omega)}$ for $\omega = \lfloor4\kappa^2\rfloor$, and $p_t = \sqrt{2\mu\eta_t}$, we have the bound
		\begin{align}\label{skipbound}
			\E{\norm{\x_T - \x_\star}^2} \leq \tfrac{8\sigma^2}{\mu^2T} + \tfrac{4\kappa^4((1+4\kappa^2)\mu^2\delta_0+4\sigma^2)}{\mu^2T^2},
		\end{align}
		implying an SFO complexity of $\O{\frac{\sigma^2}{\mu^2\epsilon} + \kappa^2\tfrac{\kappa\sqrt{\delta_0}+\sigma}{\sqrt{\epsilon}}}$ and a QMO complexity of $\O{\frac{\sigma}{\kappa + \mu\sqrt{\epsilon}} + \frac{\kappa\sqrt{\kappa\sqrt{\delta_0}+\sigma}}{\epsilon^{1/4}}}$. 
	\end{theorem}
	
	The result in Theorem \ref{ssqp-skip-thm} is the first of its kind in the context of constrained optimization, and the proof is provided in Appendix \ref{pf-ssqp-skip-thm}. The bound in Theorem \ref{ssqp-skip-thm} is even better than that obtained for the corresponding unconstrained problem in \cite{mishchenko2022proxskip}. Specifically, the proof of Algorithm \ref{ssqp-skip} proceeds in a similar manner and obtains a similar recursion as that in \cite[Lemma C.2]{mishchenko2022proxskip}. However, we utilize diminishing stepsizes to avoid the $\log(T)$ term that appears in \cite[Corollary 5.6]{mishchenko2022proxskip}.
	
	The update rule in Alg. \ref{ssqp-skip} and the statement of Theorem \ref{ssqp-skip-thm} allow us characterize the amortized per-iteration cost of SSQP-Skip. First, observe that the base cost of non-QMO iterations in Alg. \ref{ssqp-skip} is $\cO(1)$, while that of QMO-iterations, that occur with probability $p_t \sim \cO(t^{-1/2})$ is $\cO(m^3)$. Averaging over the first $T$ iterations gives the amortized expected cost of $\cO(1 + \frac{m^3}{\sqrt{T}})$. Hence, to reach $\epsilon$-close to the optimum, the amortized expected per-iteration cost becomes $\cO(1+m^3\sqrt{\epsilon})$. Compared to primal-dual methods, whose per-iteration cost is typically $\cO(m)$, SSQP-Skip is faster per-iteration when $m \leq \epsilon^{-1/4}$. 
	
	\section{Variance-Reduced Accelerated SSQP Algorithm} \label{sec-avrssqp}
	In this section, we focus on the special case of \eqref{mainProb} where $f$ has a finite-sum structure and $i_t \in \{1, \ldots, n\}$ for moderately large $n$. In the finite-sum case, we show that variance reduction can be applied to Algorithm \ref{ssqp} so as to obtain an improved dependence of the SFO complexity on $\epsilon$. In particular, we build upon the VARAG algorithm from \cite{lan2019unified} to propose the novel accelerated variance-reduced SSQP (VARAS) algorithm. As we shall show later, the performance of the proposed algorithm is also similar to that of VARAG. 
	
	The proposed updates are summarized in Algorithm \ref{var-red-ssqp}, and follow a similar structure as that of VARAG, except that the proximal step is replaced with a constrained minimization step similar to \eqref{proxlin-x}. The proposed algorithm entails several passes over the data. At the $s$-th epoch or pass, the algorithm needs the full gradient $\nabla f(\tx_{s-1})$, which is used to correct the stochastic gradient of each data point. Specifically, the $t$-th iteration of the $s$-th epoch utilizes a random $i_t \in \{1, \ldots, n\}$ and  $\nabla f(\tx_{s-1})$ to construct an unbiased estimate $\nt_t$ of $\nabla f(\y_t)$ with a variance that decreases with $s$. 
	
	We remark that although there exist several variance-reduced and accelerated stochastic optimization algorithms, we specifically selected VARAG, given its good performance and a flexible structure that allows for easy modifications. Indeed, both Katyusha acceleration \cite{allen2017katyusha} as well as related  negative momentum acceleration techniques, such as those in \cite{shang2018asvrg}, cannot be applied here for the general convex case, as they utilize a penalty parameter within the proximal operator that is required to be small or diminishing. For instance, Katyusha uses a penalty parameter $\alpha_s \sim \O{1/s}$ where $s$ is the epoch index \cite{allen2017katyusha}, and likewise, ASVRG uses $\beta_s \sim \O{1/s}$. Hence the minimization subproblem at each iteration would only be $\O{1/s}$-strongly convex, which would not be sufficient to counter the term arising from the application of the quadratic upper bound \eqref{qub} on $g_k$. Also note that while the VRADA algorithm proposed in \cite{song2020variance} achieves the best known SFO complexity in the unconstrained setting, it uses a recursively defined estimate sequence, which cannot be extended to the present setting because of the penalty term. 
	
	\begin{algorithm}
		\caption {VARAS: VAriance-Reduced Accelerated SSQP}
		\begin{algorithmic}[1]
			\State\textbf{Input:} $\x_0\in \cK \subset\Rn^d$, $\gamma = \bt/\nu$, $T_s$, $\alpha_s$, $\omega_s$, $\beta_s$, and $\theta_t$
			\State Set $\tx_0 = \z_0 = \x_0$
			\State\textbf{for} $s =1,2,...,S$
			\State\hspace{3mm} Calculate $\nabla f(\tx_{s-1})$
			\State\hspace{3mm} Set $\x_0 = \tx_{s-1}$ 
			\State\hspace{3mm} \textbf{for} $t =1,2,...,T_s$
			\State\hspace{6mm} Sample $i_t$ randomly from $\bc{1,...,n}$
			\State\hspace{6mm} Update
			\noindent\begin{align}
				\y_t &= \frac{(1+\mu\beta_s)(1-\alpha_s - \omega_s)\x_{t-1} + \alpha_s\z_{t-1}}{1+\mu\beta_s(1-\alpha_s)} \nonumber\\
				& \hspace{2cm}+ \frac{(1+\mu\beta_s)\omega_s}{1+\mu\beta_s(1-\alpha_s)}\tx_{s-1}\label{avr-yup}\\
				\z_{t-1}^+ &= \frac{1}{1+\mu\beta_s}\br{\z_{t-1} + \mu\beta_s\y_t}\label{avr-zplsup}\\
				\nt_t &= \nabla f_{i_t}(\y_t) - \nabla f_{i_t}(\tx_{s-1}) + \nabla f(\tx_{s-1}) \label{avr-grad-up}
			\end{align}
			\begin{align}
				\z_t = \argmin_{\u \in \Rn^d} &\Bigl\{\alpha_s\beta_s\br{\ip{\nt_t}{\u} + \frac{\mu}{2}\norm{\y_t - \u}^2 +  h(\u)} \nonumber\\
				&+ \frac{\alpha_s}{2}\norm{\z_{t-1}-\u}^2  \label{avr-zup}\\
				&\hspace{-1cm}+ \gamma\beta_s \max\{[g_k(\y_t) + \alpha_s\ip{\nabla g_k(\y_t)}{\u - \z_{t-1}^+}]_+\}\Bigr\} \nonumber
			\end{align}
			\begin{align}
				\x_t &= (1-\alpha_s - \omega_s)\x_{t-1} + \alpha_s \z_t + \omega_s \tx_{s-1} \label{avr-xup}
			\end{align}
			\State\hspace{3mm} \textbf{end for}
			\State\hspace{3mm} Set $\tx_s = \br{\sum_t \theta_t \x_t}/\sum_t\theta_t$
			\State\hspace{3mm} Reset $\z_0 = \z_{T_s}$
			\State\textbf{end for}
			\State\textbf{Output:} $\bar{\x} = \tx_S$.
		\end{algorithmic}
		\label{var-red-ssqp}
	\end{algorithm} 
	
	The following theorem, whose proof is provided in the supplementary material, establishes the oracle-complexity bounds for VARAS.
	\begin{theorem}\label{avr-thm}
		Under Assumptions \ref{slater}-\ref{init}, let $L_\gamma := L_f + \gamma L_g$, $\kappa := L_\gamma/\mu$, 
		$D_0  := 2B_{\gamma} + \frac{3L_\gamma}{2}B_x \geq  2\br{\EE F(\tx_0) - F(\x_\star)} + \tfrac{3L_\gamma}{2}\EE\norm{\z_{0} - \x_\star}^2$, $s_0:= \lfloor\log n\rfloor + 1$, $\beta_s = \frac{1}{3\alpha_sL_\gamma}$, $\omega_s = \frac{1}{2}$, and $T_s = 2^{s-1}$ for $s \leq s_0$ and $T_s = T_{s_0}$ for $s > s_0$. Then, we have the following oracle complexity bounds. 
		\begin{enumerate}
			\item When $f_i$ are convex and $\alpha_s = \min\{\tfrac{1}{2},\frac{2}{s-s_0+4}\}$ and
			\begin{align}
				\theta_t &= \begin{cases}
					\frac{\beta_s}{\alpha_s}(\alpha_s+\omega_s)& 1\leq t\leq T_s-1\\
					\frac{\beta_s}{\alpha_s}& t = T_s ,
				\end{cases}\label{thetacond}
			\end{align}
			then the oracle complexity of Algorithm \ref{var-red-ssqp} is given by 
			\begin{align}
				N_{\text{SFO}} &= \begin{cases}
					\cO\br{n\log{\frac{D_0}{\eps}}}& n \geq \frac{D_0}{\epsilon},\\ 
					\cO\br{n\log{n} + \sqrt{\frac{nD_0}{\epsilon}} } & n < \frac{D_0}{\epsilon},
				\end{cases}\label{avr-c-res}\\
				N_{\text{QMO}} &= \begin{cases}
					\cO\br{\frac{D_0}{\eps}}& \hspace{1.5cm} n \geq \frac{D_0}{\epsilon},\\ 
					\cO\br{\sqrt{\frac{nD_0}{\epsilon}} } & \hspace{1.5cm} n < \frac{D_0}{\epsilon}.
				\end{cases}\label{avr-c-res-q}
			\end{align} 
			\item When $f_i$ are $\mu$-strongly convex, let $\alpha_s = \min\{\tfrac{1}{2},\max\bc{\frac{2}{s-s_0+4}, \min\bc{\sqrt{\frac{n}{3\kappa }},\frac{1}{2} } }\}$ and $\theta_t$ is set as in \eqref{thetacond} if $1\leq s \leq s_0$ or $s_0< s\leq s_0 + \sqrt{\frac{12\kappa}{n}} -4$, $n < \frac{3\kappa}{4}$. Otherwise set as 
			\begin{align}
				\theta_t = \begin{cases}
					\Gamma_{t-1} - (1-\alpha_s-\omega_s)\Gamma_t& 1\leq t\leq T_s-1\\ 
					\Gamma_{t-1}& t = T_s ,
				\end{cases}
			\end{align} 
			where $\Gamma_{t} = \br{1+\mu\beta_s}^t$. Then, the oracle complexity is given by
		\end{enumerate}
		\begin{align}
			N_{\text{SFO}} &= \begin{cases}
				\cO\br{n\log{\frac{D_0}{\eps}}} \hspace{10mm} n \geq \frac{D_0}{\epsilon} \text{ or } &n \geq \frac{3\kappa}{4},\\ 
				\cO\br{n\log{n} + \sqrt{\frac{nD_0}{\epsilon}} } & n < \frac{D_0}{\epsilon} \leq \frac{3\kappa}{4}, \\
				\cO\br{n\log{n} + \sqrt{n\kappa}\log{\frac{4D_0}{3\kappa\epsilon}}} & n < \frac{3\kappa}{4} \leq \frac{D_0}{\epsilon},
			\end{cases}\\
			N_{\text{QMO}} &= \begin{cases}
				\cO\br{\frac{D_0}{\eps}} &n \geq \frac{D_0}{\eps},\\ 
				\cO\br{n\log\tfrac{D_0}{\eps}}&\frac{3\kappa}{4}< n \leq\frac{D_0}{\eps} ,\\
				\cO\br{\sqrt{\frac{nD_0}{\epsilon}} } & n < \frac{D_0}{\epsilon} \leq \frac{3\kappa}{4} ,\\
				\cO\br{n\log{n} + \sqrt{n\kappa}\log{\frac{4D_0}{3\kappa\epsilon}}} & n < \frac{3\kappa}{4} \leq \frac{D_0}{\epsilon}.
			\end{cases}
		\end{align}
	\end{theorem}
	
	The proof of Theorem \ref{avr-thm} begins by deriving a one-step inequality based on the update in \eqref{avr-zup} and the relationships among the parameters. The resulting inequality (see the supplementary material) matches the form of \cite[Lemma 6]{lan2019unified}, so the remaining arguments in \cite[Theorems 1–2]{lan2019unified} apply directly. Though the rates established in Theorem \ref{avr-thm} are the fastest, the most common case is when $\epsilon$ is small and VARAS achieves SFO complexity of $\O{\frac{1}{\sqrt{\epsilon}}}$ and $\O{\log\br{\frac{1}{\epsilon}}}$ for the convex and strongly convex cases, respectively. These rates are clearly superior to the best known rates for constrained problems in  \cite{xu2021iteration,lin2018levelDet,yu2017simple}. Before concluding the theoretical results, the following remark is due. 
	\begin{rem}
		\label{rem_flop}
		The chosen oracle model hides the computational costs for large $m$, since each SFO call returns $\{\nabla f_i(\x), \{g_k(\x), \nabla g_k(\x)\}_{k=1}^m\}$. Indeed, solving the QP at each iteration of SSQP may incur $\cO(m^3)$ floating point operations (flops), which is significantly more than the usual $\cO(m)$ flops incurred by similar primal--dual algorithms. While SSQP-Skip does allay this concern to an extent, it remains an open problem to see if we can design algorithms that work with only one (or a few) of the constraints at every iteration. 
	\end{rem}

	\section{Numerical Experiments}\label{sec-numerical}
	
	In this section, we analyze the performance of our proposed algorithms on two real-world problems and compare them with Adaptive Primal-Dual SGD (APriD) \cite{yan2022adaptive}, Generalized Online Convex Optimization (GOCO) \cite{yuan2018online}, primal--dual stochastic subgradient (PDSS) method \cite{madavan2021stochastic}, Stochastic Subgradient Projection (SSP)  \cite{necoara2022stochastic}, and SGD by one projection by a smoothing technique (SGDP-ST) \cite{mahdavi2012stochastic}. We also include RECOO from \cite{guo2022online} as a representative proximal baseline, since the prox-linear structure of our proposed algorithms makes it a relevant benchmark, particularly for SSQP and SSQP-Skip. As RECOO is designed for single-constraint problems, we adapt it to our setting using the standard aggregation $g(\x)=\max_k g_k(\x)$ \cite{mahdavi2012stochastic}. Each iteration then requires solving a general convex subproblem, leading to a substantially higher per-iteration cost than the QP subproblems used by SSQP and SSQP-Skip. Other older algorithms in Table \ref{constr-papers} are not included here as they were not directly comparable. For instance, the problem in  \cite{lan2020algorithms} special case of \eqref{mainProb} as it consider only a single constraint, while \cite{lin2018level} focuses on a nested finite-sum structure different from \eqref{mainProb}. The works in \cite{mahdavi2012stochastic} and \cite{bayandina2018mirror} did not provide any numerical results, making it difficult to tune hyperparameters and adapt these methods to the problems considered here.
	
	We remark that the purpose of this section is to illustrate the effects of constraints and examine the trade-offs between the solvers. These benchmarks are not exhaustive, since we do not carry out many scaling (i.e. examining the effect of $n$, $\kappa$) or other ablation studies. Such studies are outside the scope of the current work, and we do not expect them to yield any new insights beyond what we already know from existing literature. All experiments were run in MATLAB R2023a (Intel Core i7, 16 GB RAM), using \texttt{quadprog} for the intermediate QPs and CVX to solve the general convex optimization subproblems.
	
	\subsection{Trajectory generation for an unmanned surface vehicle}
	Here we consider Zermelo's navigation problem \cite{zermelo1931navigationsproblem} in an oceanic environment where the aim is to find two-dimensional energy-optimal trajectory for an unmanned surface vehicle (USV) operating in a rectangular region $\cX := \{\x \mid \norm{\x-\u}_\infty \leq r\}$. Let $\x(t) \in \Rn^2$ denote the position of the USV at discrete time $t \in \{1, \ldots, T\}$. Since the USV operates in a small and homogeneous area, we model the surface current velocity at coordinate $\y$ as an unknown linear function $\v(\y) = \W\y+\z$. However, exact information on the ocean current at each position is unavailable; rather, several oceanographic agencies~\cite{mercator, copernicus_marine} publish estimated measurements. As considered in \cite{yooEnsemble}, we seek to find an energy-efficient USV trajectory given the ensemble of ocean current estimates, denoted by $\{\W_i,\z_i\}_{i=1}^n$. 
	
	The energy consumption for a USV to move from $\x(t-1)$ to a nearby point $\x(t)$ scales cubically with the effective speed and can be modeled as $\norm{\x(t-1)-\x(t)-\v(\x(t))}^3$~\cite{jones2017planning}. Our goal is to minimize the total energy. If the maximum velocity of a USV is $s_{\max}$ and the maximum surface current speed $s_w \ll s_{\max}$, then the constraint $\norm{\x(t-1)-\x(t)} \leq v_{\max}:=s_{\max}-s_w$ ensures that the generated trajectories are feasible for the lower level controller. Therefore, the minimum expected-energy trajectory from the starting position $\p_{\text{start}}$ to the destination position $\p_{\text{dest}}$, where $\x$ collects all the coordinates across $T$ discrete time instances, is the solution to
	\begin{align}
		\min_{\x \in \Rn^{2T}} \frac{1}{n} &\sum_{i=1}^{n}\sum_{t=2}^{T}\norm{\x(t-1)-\x(t)-\W_i\x(t-1)-\z_i}^3 \nonumber\\
		\text{s. t. } & \x(1) = \p_{\text{start}}, \hspace{5mm} \x(T) = \p_{\text{dest}} 	\label{eq:USV_opti_prob}\\
		\hspace{-4mm} &\norm{\x(t-1)-\x(t)}^2 \leq v^2_{\max},  ~~2\leq t \leq T,\nonumber
	\end{align}
	which has the familiar finite-sum structure with convex objective and constraints as in \eqref{mainProb}.  
	
	To generate an ensemble $\{\W_i,\z_i\}_{i=1}^n$, we first randomly generate $\W\in \Rn^{2\times 2}$ and $\z\in \Rn^2$. Then for each $i \in \{1, \ldots, n\}$, we generate noisy velocities $\v_i(\y_j)$ as
	\begin{align}
		\label{noisy_velocity}
		\v_i(\y_j) &= (\I+\text{diag}(\vxi))(\W\y_j + \z)& j &= 1, 2, 3
	\end{align}
	at the three sample positions $\{\y_j\}_{j=1}^3$,  where, $\vxi \sim \cN\br{\zz, \I}$, and solve the system of equations in (\ref{noisy_velocity}) to get $\{\W_i,\z_i\}$. We consider a $200 \times 200$ square region (in arbitrary units). The USV needs to travel from $(20,20)$ to $(180, 180)$. We consider two different settings, (a) $n=100$, $T=40$, with speed limit $v_{max} = 8.1$ unit/s; and (b) $n=1000$, $T=100$ with speed limit $v_{max}=3.3$ unit/s. We initialized all the algorithms with the straight-line path joining $\p_{\text{start}}$ and $\p_{\text{dest}}$ with $T$ equidistant waypoints.
	
	\captionsetup[subfloat]{labelformat=empty}
	\begin{figure}[ht]
		\centering
		\subfloat{
			\includegraphics[width=0.23\textwidth]{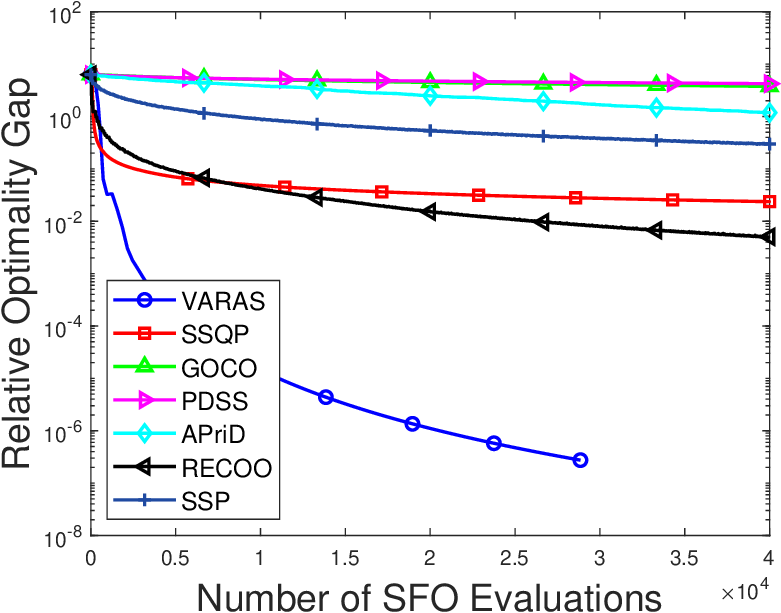}
		}
		\hfill
		\subfloat{
			\includegraphics[width=0.23\textwidth]{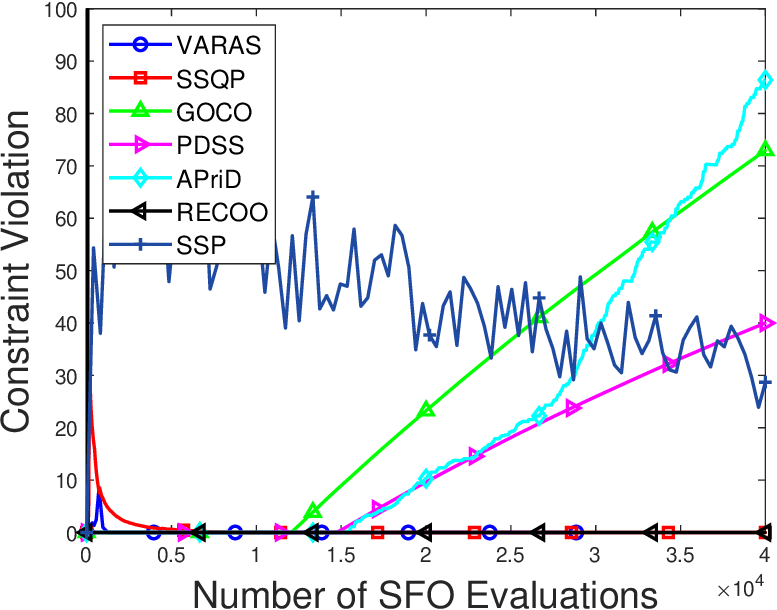}
		}
		\vspace{0.05cm}
		\subfloat{
			\includegraphics[width=0.23\textwidth]{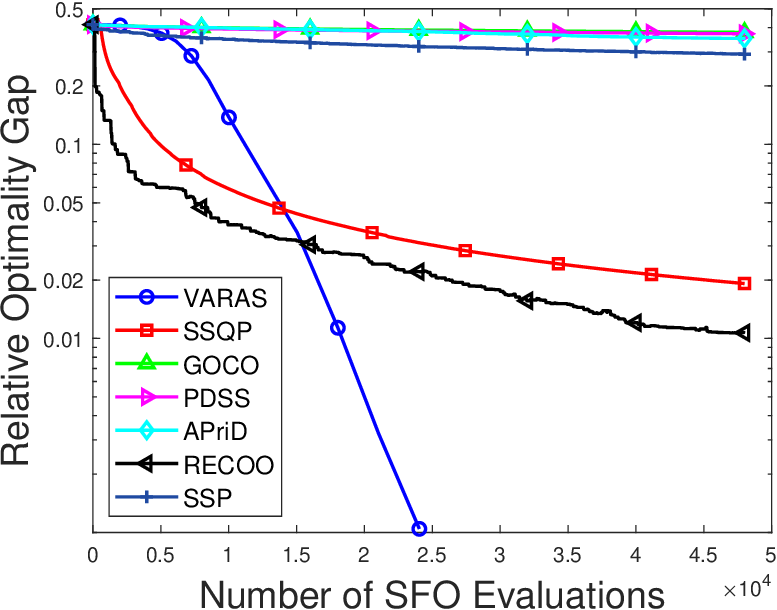}
		}
		\hfill
		\subfloat{
			\includegraphics[width=0.23\textwidth]{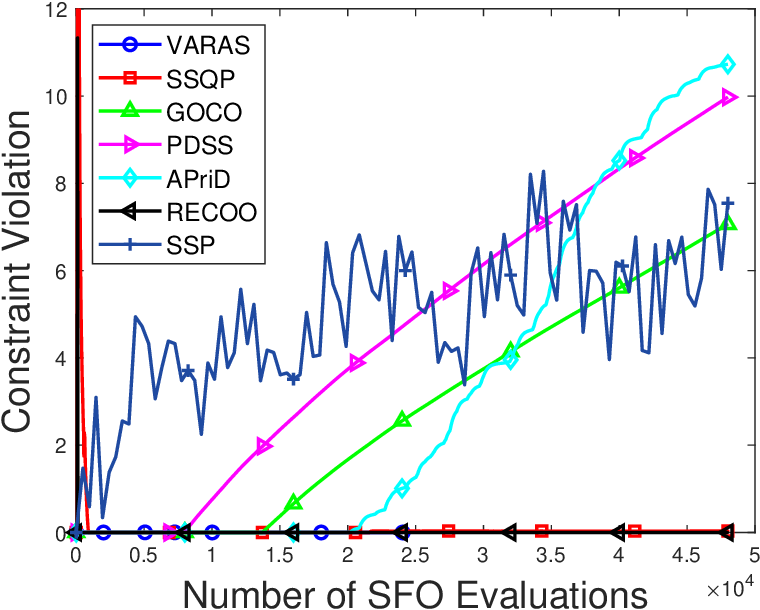}
		}
		\caption{Comparison of different methods with respect to the number of SFO evaluations. The top row corresponds to the setting (a) $n=100$ and $T=40$, while the bottom row corresponds to the setting (b) $n=1000$ and $T=100$.}
		\label{fig:comparison_usv}
	\end{figure}
	\begin{figure}[ht]
		\centering
		\includegraphics[width=0.55\textwidth]{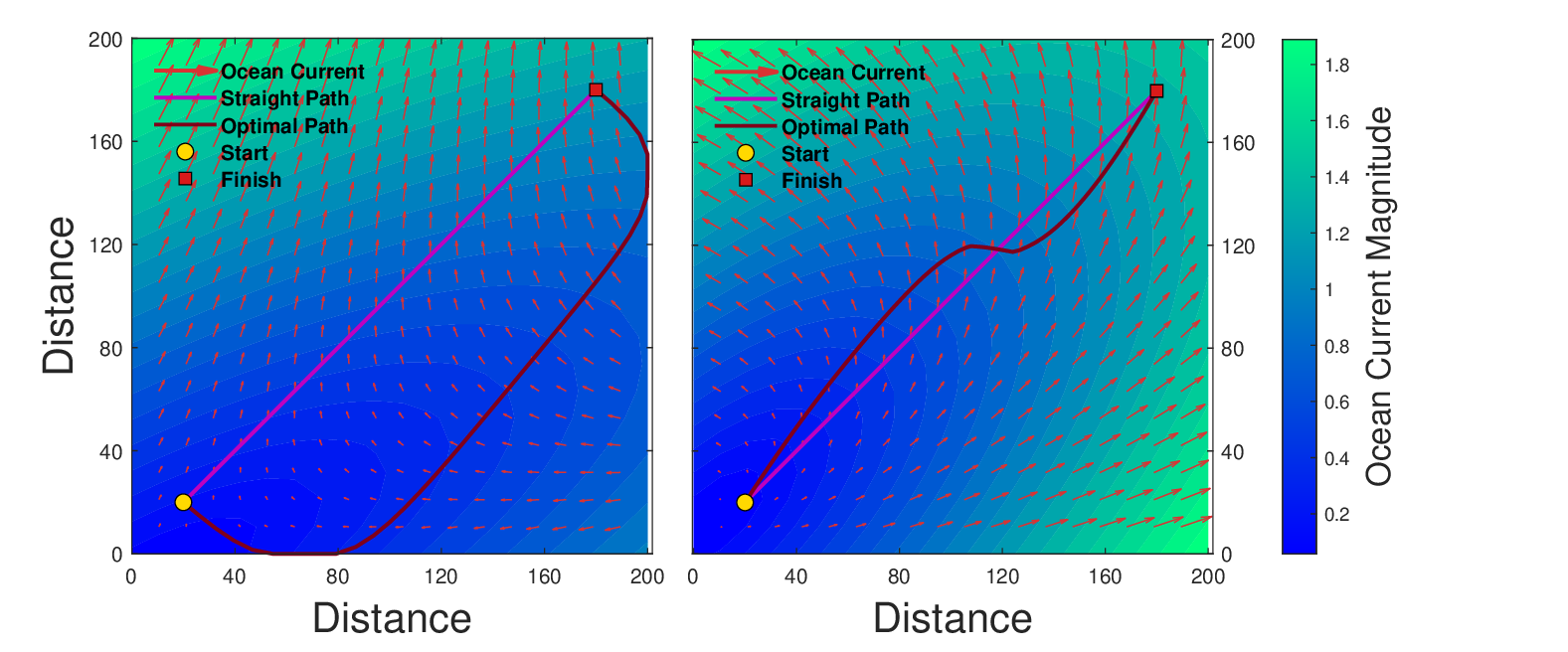}
		
		\caption{Straight path and optimal trajectories obtained VARAS. The left plot corresponds to setting (a) $n=100$, $T=40$ while the right plot corresponds to setting (b) $n=1000$, $T=100$.}
		\label{fig:optimal_path}
	\end{figure}
	
	
	\begin{table}[t]
		\centering
		\caption{Hyperparameter settings used in the trajectory generation experiments. The parameters $\eta_0$ and $\eta_{\max}$ are specific to SSQP, while $L_\gamma$ is specific to VARAS.}
		\label{tab:traj_hyperparams}
		\begin{tabular}{c c c c c}
			\hline
			Setting & Algorithm & $\eta_0$, $\eta_{max}$ / $L_\gamma$ & $\gamma$ & Mini-batch Size \\
			\hline
			\multirow{2}{*}{(a)}
			& SSQP & $0.009$, $0.009$ & $6\times10^{5}$ & $4$ \\
			& VARAS & $350$ & $10^{6}$ & -- \\
			\hline
			
			\multirow{2}{*}{(b)}
			& SSQP & $10^{-3}$, $5\times 10^{-5}$ & $10^{9}$ & $8$ \\
			& VARAS & $25\times10^{3}$ & $10^{12}$ & -- \\
			\hline
		\end{tabular}
	\end{table}
			%
	
			%
			%
	
	\begin{figure}[ht]
		\centering
		\includegraphics[width=0.4\textwidth]{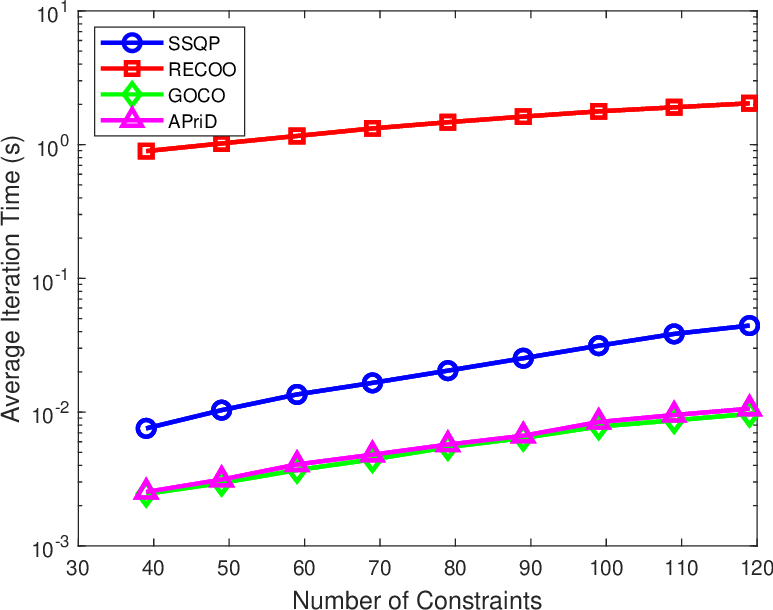}
		\caption{Average iteration time versus the number of constraints for $n=1000$. The parameter $T$ is varied from $40$ to $120$ in increments of $10$.}
		\label{fig:scaling_with_constraints}
	\end{figure}
	
	\begin{figure}[ht]
		\centering
		\subfloat[]{
			\includegraphics[width=0.23\textwidth]{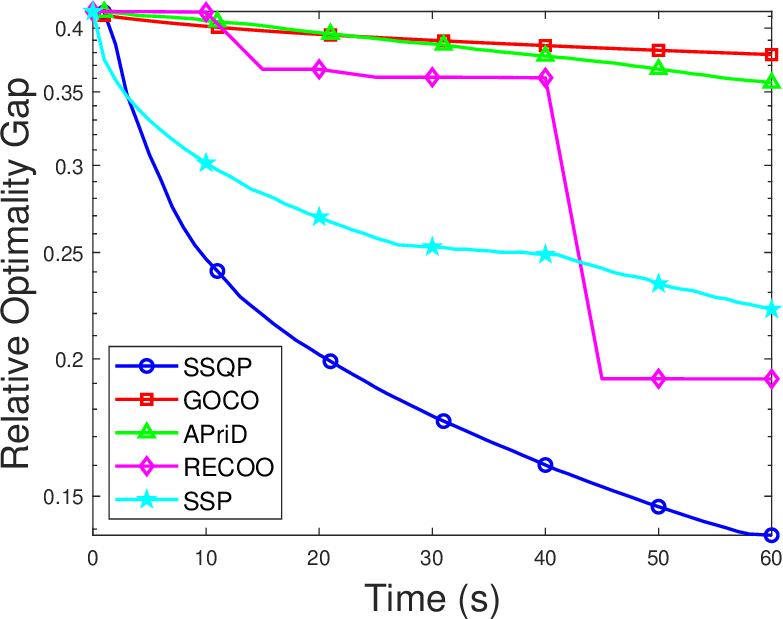}
			
		}
		\hfill
		\subfloat[]{
			\includegraphics[width=0.23\textwidth]{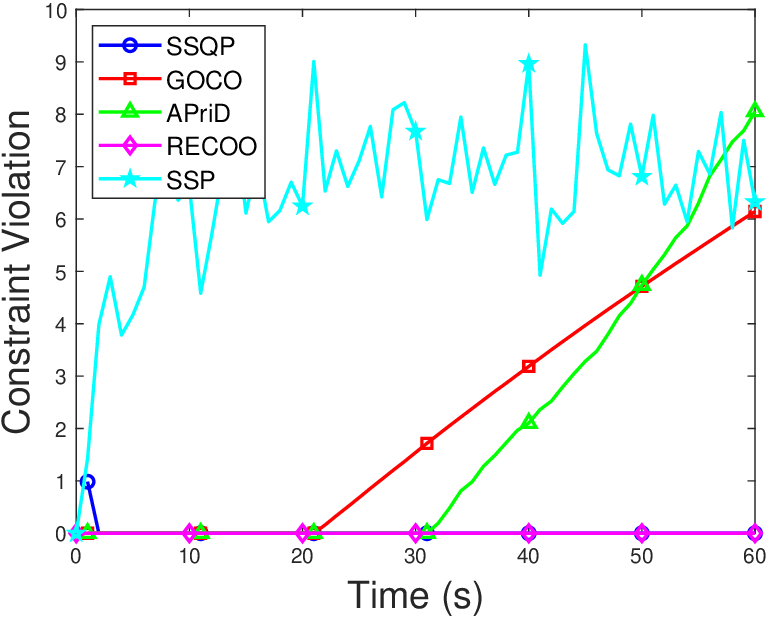}
			
		}
		\caption{Comparison of convergence behavior in terms of wall-clock time for $n=1000$, $T=100$.}
		\label{fig:conv_with_time}
	\end{figure}
	
	\subsubsection{Hyperparameters} The hyperparameters of each algorithm are tuned to achieve the best performances.  For SSQP, we used the learning rate schedule $\eta_t = \min\{\eta_0/\sqrt{t}, \eta_{\max}\}$. The hyperparameter values that yielded the best performance are reported in Table~\ref{tab:traj_hyperparams}. A general trend, observed across all algorithms, was the apparent trade-off between optimality gap and constraint violation. A choice of hyperparameters that improved one worsened the other. We observed that APriD, GOCO, and PDSS diverged once the iterates left the feasible set. We also found that the constraint violation of SSP showed highly irregular behavior, often decreasing and increasing across iterations without a clear trend. These behaviors persisted across the range of hyperparameters we tested. Therefore, for the above-mentioned methods, we selected hyperparameters that provided a reasonable balance between reducing the optimality gap and controlling constraint violation. For SSQP and RECOO, the hyperparameters were tuned so that the constraint violation decreased to near zero within a few iterations. In contrast, for VARAS, the hyperparameters were selected to ensure exact feasibility, i.e., almost zero constraint violation. The value of $F(\x_\star)$ was computed after running $10^6$ iterations of SSQP, and we compare all the algorithms in terms of the relative optimality gap, i.e. $\frac{F(\x)-F(\x_\star)}{F(\x_\star)}$. 
	
	\subsubsection{Optimality gap and constraint violation} Fig.~\ref{fig:comparison_usv} shows the relative optimality gap and constraint violation of different algorithms as a function of the number of SFO calls for two different experimental settings. We observe that, excluding RECOO, SSQP demonstrates superior performance compared to the other baseline algorithms; moreover, the performance gap between SSQP and RECOO remains small. The slightly better performance of RECOO is not unexpected, as it is a proximal-based method that solves a more general convex optimization subproblem at each iteration, whereas SSQP solves a quadratic program based on local first-order approximations. Overall VARAS outperforms all other algorithms by a substantial margin in both cases. Fig.~\ref{fig:optimal_path} shows the least-energy trajectories obtained by VARAS. The optimal trajectories achieve energy reduction of approximately $86\%$ for setting (a) and $30\%$ for setting (b) relative to the straight-line trajectories.
	
	\subsubsection{Per-iteration runtime} 
	
	Next, we examine the per-iteration runtime of some of the representative algorithms. Since SSQP requires solving a QP subproblem whose worst-case computational complexity is $\mathcal{O}(m^3)$, where $m$ denotes the number of constraints, we investigate the effect of increasing the number of constraints on the per-iteration cost by varying the number of waypoints $T$ from $40$ to $120$, while keeping $n$ fixed at $1000$. Fig.~\ref{fig:scaling_with_constraints} compares the average per-iteration cost of SSQP against the primal-dual methods GOCO and APriD, as well as RECOO, a general proximal optimization method. We note that SSP is not included in this comparison since it randomly selects one constraint at each iteration, and therefore its per-iteration cost does not scale with the total number of constraints. As expected, SSQP is faster than RECOO, which solves a general convex problem at each iteration, but slower than APriD and GOCO, whose constraint-related computations require only $\cO(m)$ operations per-iteration. Interestingly, while SSQP is only 2--3 times slower than APriD and GOCO, it is almost 100 times faster than RECOO, supporting our earlier observation that solving the QP subproblem can be substantially simpler than solving a general convex subproblem in this setting.
	
	\subsubsection{Overall runtime}
	The preceding sections show that SSQP has a higher per-iteration cost than primal-dual methods, but also requires fewer iterations than others to converge. We therefore next compare the algorithms in terms of wall-clock time, which captures both iteration complexity and per-iteration cost. Specifically, we compare four representative classes of methods: the proposed SSQP algorithm, primal-dual methods (GOCO and APriD), the stochastic subgradient projection method SSP, and the proximal method RECOO. We use setting (b), corresponding to $n=1000$ and $T=100$. Fig.~\ref{fig:conv_with_time} shows that SSQP achieves the best overall wall-clock performance among the tested methods, indicating that its faster convergence more than compensates for the additional cost of solving QP subproblems in this regime.

	\begin{figure}[ht]
		\centering
		\subfloat[]{
			\includegraphics[width=0.23\textwidth]{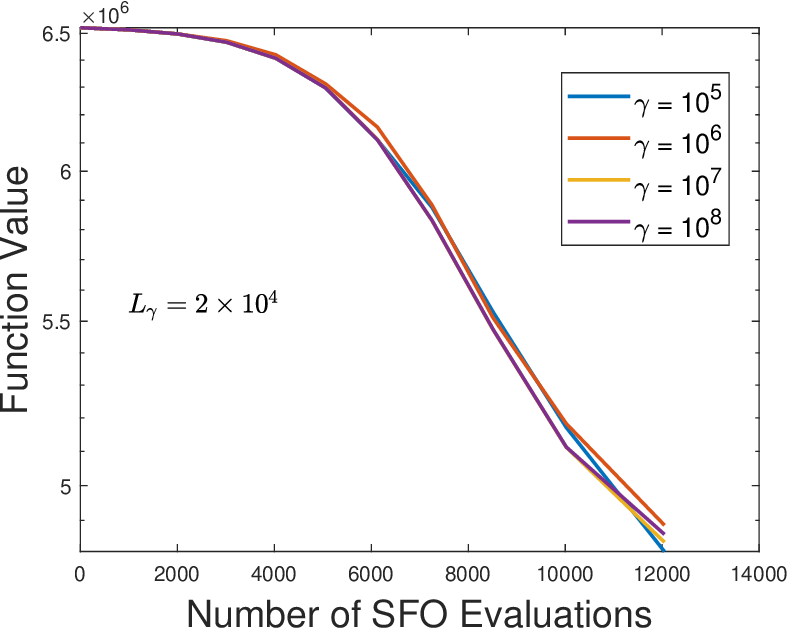}
			
		}
		\hfill
		\subfloat[]{
			\includegraphics[width=0.23\textwidth]{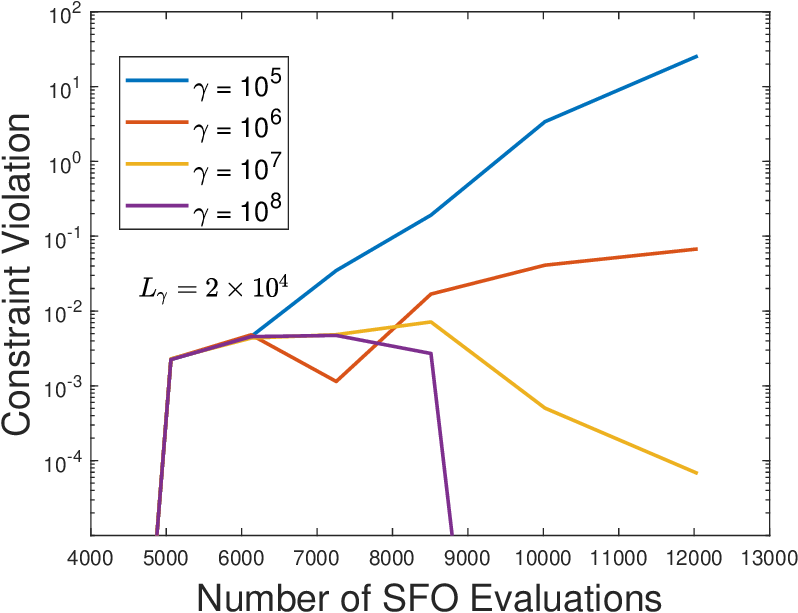}
		}
		\vspace{0.05cm}
		\subfloat[]{
			\includegraphics[width=0.23\textwidth]{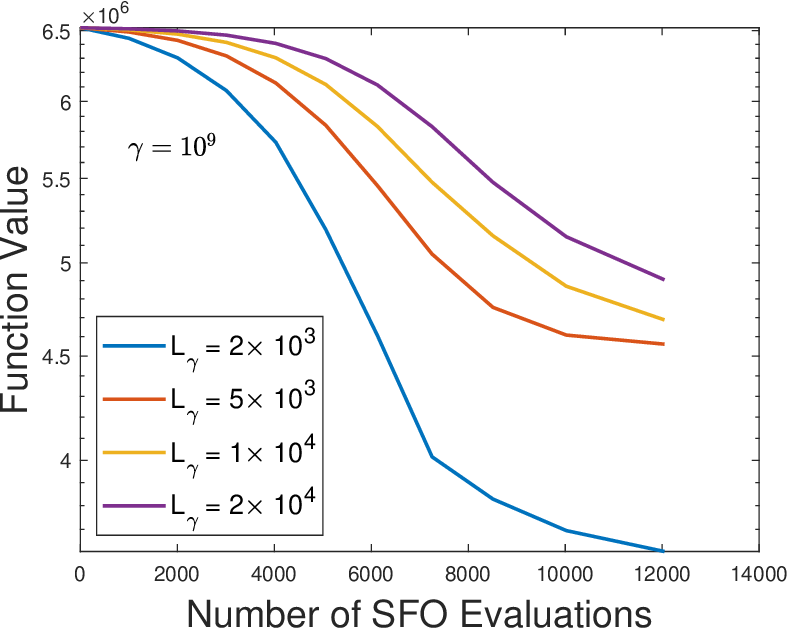}
			
		}
		\hfill
		\subfloat[]{
			\includegraphics[width=0.23\textwidth]{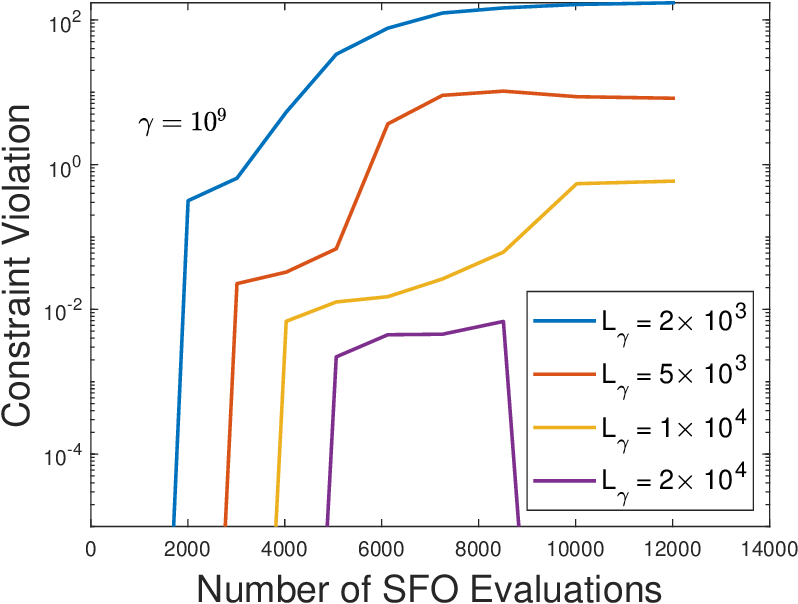}
		}
		\caption{ Sensitivity of VARAS to hyperparameters. The top row shows the 
			effect of $\gamma$ on VARAS with $L_\gamma=2\times 10^4$. Larger values of $\gamma$ improve feasibility. The bottom row illustrates the effect of $L_\gamma$ on VARAS for $\gamma=10^9$. Larger values of $L_\gamma$ result in slower convergence but yield better feasibility performance.}
		\label{fig:varas_sensitivity}
	\end{figure}
	
	\subsubsection{VARAS sensitivity to hyperparameters}
	Since the theoretically sufficient values of $\gamma$ and $L_\gamma$ may be conservative, we include a sensitivity study for these parameters. The penalty parameter $\gamma$ controls the extent to which constraint violations are penalized during the intermediate iterations. Theoretically, $\gamma$ is required to be sufficiently large, i.e., $\gamma > \tilde{B}/\nu$. In practice, across both experiments and for all three variants, namely SSQP, SSQP-Skip, and VARAS, we observed that once $\gamma$ exceeded a problem-dependent threshold, increasing it further did not deteriorate the performance. The top row of Fig.~\ref{fig:varas_sensitivity} illustrates the effect of $\gamma$ on VARAS for a fixed value of $L_\gamma$. As can be seen, the objective value evolves almost identically for different choices of $\gamma$. On the other hand, the constraint violation is large only when $\gamma$ is too small, and remains close to zero otherwise.
	
	The parameter $L_\gamma$ can in principle be approximated as in the original VARAG algorithm \cite{lan2019unified}, for instance using curvature information of the objective function. However, we found it simpler to tune $L_\gamma$ directly, similar to how the learning rate is routinely tuned in stochastic gradient methods. Empirically, $L_\gamma$ is observed to play a role analogous to an inverse learning-rate parameter. This behavior is evident from the bottom row of Fig.~\ref{fig:varas_sensitivity}, where the objective value decreases more slowly as $L_\gamma$ increases. At the same time, choosing $L_\gamma$ too small leads to large constraint violations, whereas larger values yield near-zero constraint violations.
	
	\subsection{Regression with Residual Constraints}
	Regression is a fundamental tool in signal processing and learning. However, in many practical applications, e.g. in wireless communications \cite{song2019set}, in addition to simply fitting a regression model to the observed data, it is desirable to ensure that, for some critical samples, the loss remains below a prescribed tolerance. We can write the constrained regression problem as
	\begin{align}
		\label{eq:regression_opti_prob}
		&\min_{\th \in \Rn^{d}} \frac{1}{2n} \sum_{i=1}^{n}\ell(y_i, b_\th(\x_i)) \\
		&\text{s. t. } \ell(y_k, b_\th(\x_k)) \leq r,  ~~k \in \bc{1,2,\ldots,K},\nonumber
	\end{align}
	where $\ell$ denotes the loss function,  $b_\th$ denotes the regression model, and $\bc{(\x_i, y_i)}_{i=1}^n$ are the data points whose first $K$ tuples belong to the critical set over which the loss should be below $r$. While \eqref{eq:regression_opti_prob} is a special case of \eqref{mainProb} whenever the objective and the constraints are convex, we consider linear regression for simplicity, i.e., $\ell(y_i, b_\th(\x_i)) = (y_i-\x_i^\T\th)^2$. We remark that the adaptive version of \eqref{eq:regression_opti_prob} has been widely studied within the framework of set-membership adaptive filtering \cite{bhotto2011robust, flores2019set}. 
	
	\begin{table}[ht]
		\centering
		\caption{Performance comparison for different threshold values (averaged over 50 runs).} 
		\begin{tabular}{|l| l| c| c| c|}
			\hline
			\multicolumn{2}{|c|}{ } & \multicolumn{3}{c|}{\textbf{Thresholds}} \\
			\cline{3-5}
			\multicolumn{2}{|c|}{ } & 0.005 & 0.001 & 0.0008 \\
			\hline
			\multirow[c]{3}{*}{\textbf{SSQP-Skip}} 
			& SFO    &110858   &416532   &473856   \\
			& QMO    &139   &216   &232   \\
			& Time (s)&6.61   &20.48   &23.07   \\
			\hline
			\multirow[c]{2}{*}{\textbf{GOCO}}
			& SFO    &786662   &1753799   &1956357   \\
			& Time (s)&36.48   &79.98   &88.93   \\
			\hline
			\multirow[c]{2}{*}{\textbf{APriD}}
			& SFO    &693345   &1277240   &1368064   \\
			& Time (s)&261.30   &492.56   &529.50   \\
			\hline
		\end{tabular}
		\label{tab:threshold_comparison}
	\end{table} 
	
	Here, we evaluate the performance of the proposed algorithms on the Year Prediction dataset~\cite{year_prediction_msd_203} which consists of $515345$ data points with $90$ features, resulting in the data matrix $\X$. We normalized the data and labels before conducting the experiment.
		We randomly selected $K=500$ data points to form the set of critical samples $\cD_c = \bc{\x_k, y_k}_{k=1}^{K}$. The remaining points constituted the dataset $\cD_f = \bc{\x_i, y_i}_{i=1}^{n}$. We empirically examined the optimization problem in (\ref{eq:regression_opti_prob}) and set $r = 0.49$, ensuring that the feasible set is non-empty.
	
	\begin{figure}[ht]
		\centering
		
		\subfloat[]{
			\includegraphics[width=0.23\textwidth]{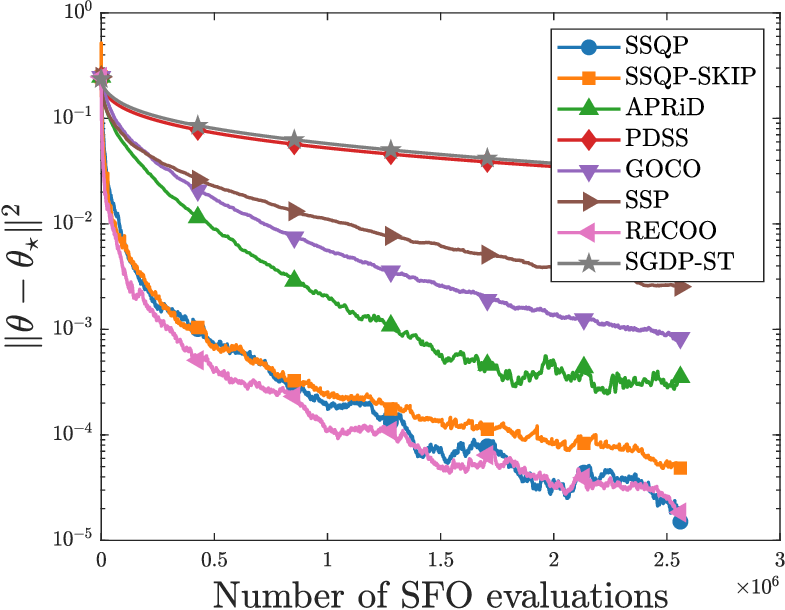}
			\label{fig:regres_fig1}
		}
		\hfill
		\subfloat[]{
			\includegraphics[width=0.23\textwidth]{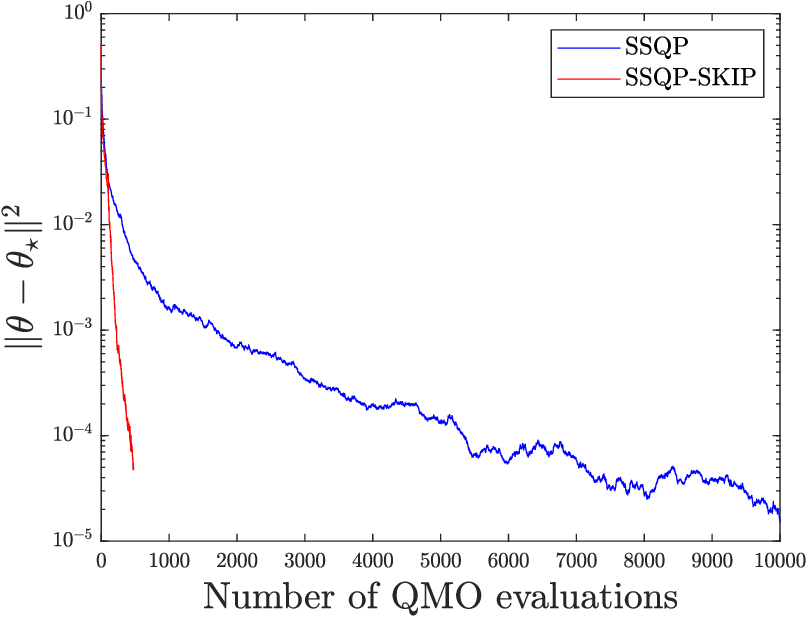}
			\label{fig:regres_fig2}
		}
		
		\caption{Squared distance to the optimum versus the SFO and QMO calls.}
		\label{fig:combined_regress}
	\end{figure}
	
	The proposed algorithms SSQP and SSQP-Skip are compared against state-of-the-art APriD, GOCO, PDSS, SSP, and RECOO algorithms, with all hyperparameters separately tuned to yield the best convergence rates. For SSQP, the configuration $L=1.1$, $\mu=0.3$, $\gamma = 10^5$ yielded the best performance, whereas for SSQP-Skip the optimal setting was $L=0.8$, $\mu=0.4$, $\gamma = 10^6$. We fixed the minibatch size to $256$ for all algorithms. We observed that in the initial $K_s$  iterations, not skipping the QMO step in SSQP-Skip empirically improved its performance, a strategy we call \emph{kickstart}. In our experiment, we set $K_s=100$ which is very small in comparison to the $10^4$ iterations required by all algorithms. For GOCO, empirically it was observed that clipping the subgradient lead to better convergence. The optimal point $\x_\star$ was obtained by running SSQP for a very large number of iterations. Fig. \ref{fig:combined_regress} shows the convergence rate in terms of the distance to the optimal point versus the number of SFO and QMO evaluations.

	Since APriD, GOCO, PDSS, and SSP do not solve quadratic programs, it is natural to quantify the advantage of SSQP-Skip in terms of wall-clock time. Table~\ref{tab:threshold_comparison} reports the SFO and QMO complexities of SSQP-Skip, APriD, and GOCO, averaged over 50 runs, required to ensure that the squared distance from the optimum is at most $\epsilon \in \{0.005, 0.001, 0.0008\}$. We chose APriD and GOCO for wall-clock time comparison as these algorithms showed better empirical performances than PDSS and SSP (refer Fig.~\ref{fig:combined_regress}).We note that APriD also requires a matrix-vector multiplication at every iteration to compute the Lagrangian derivative, resulting in higher wall-clock time despite its superior SFO complexity, observed empirically, compared to GOCO. As evident from the results, SSQP-Skip achieves better performance in terms of both SFO complexity and wall-clock time.
	
	\begin{figure}[ht]
		\centering
		\includegraphics[width=0.4\textwidth]{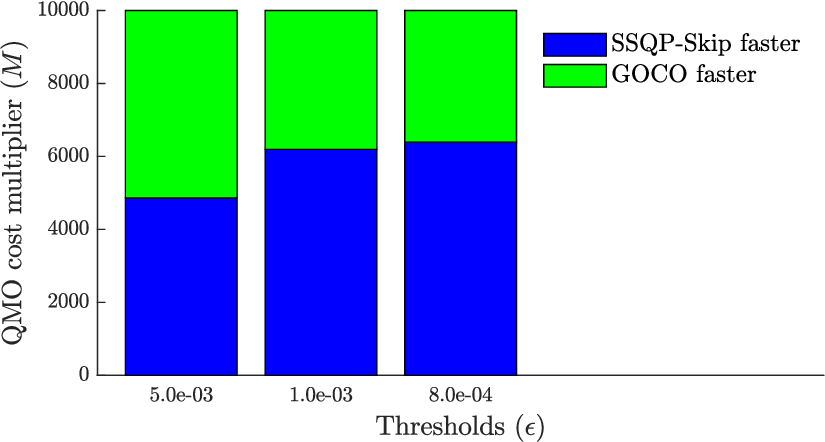}
		\caption{SSQP-Skip is faster when $M$ is small (blue region) and slower when $M$ is large (green region). }
		\label{fig:runtime}
	\end{figure}
	
	Raw CPU times on modern hardware may vary widely due to differing implementations of quadratic optimization across commercial solvers and open-source libraries. To abstract away these implementation effects, we assume an SFO that returns stochastic gradient information in $\tau$ seconds and a QP solver that completes each intermediate quadratic program in $M\tau$ seconds, where $M \geq 1$ captures the additional cost of solving a QP relative to an SFO-only setting. With this, the wall-clock time of SSQP-like algorithms is $\cO\br{\text{SFO}+M \times \text{QMO}}$, where SFO and QMO denote the number of SFO and QMO calls, respectively. Intuitively, if $M$ is very large, even a few QMO calls will hurt the wall-clock time and SSQP-Skip will take longer than other algorithms. Figure~\ref{fig:runtime} illustrates the critical values of $M$ below which SSQP-Skip remains faster than GOCO, based on the complexities summarized in Table~\ref{tab:threshold_comparison}. We remark that $M$ also depends on minibatch sizes.  

	\section{Conclusion}\label{sec-conclusion}
	This paper proposes the stochastic sequential quadratic programming (SSQP) framework, where each iteration requires solving a quadratic program (QP) with linearized constraints. For the convex and strongly convex cases, SSQP achieves rates on par with those of unconstrained stochastic gradient descent. Additionally, we propose the SSQP-Skip algorithm which requires solving QPs only on a small subset of iterations, resulting in reduced wall-clock times. For the finite-sum case, we propose the accelerated variance-reduced VARAS algorithm that also achieves near-optimal iteration complexity, improving upon existing results for constrained problems. The performance of the proposed algorithms, tested on trajectory generation and constrained regression problems, is also significantly better than the related primal--dual and other approaches in the literature. 
	
	\appendices
	\section{Basic inequalities}\label{basic}
	This section details some basic inequalities that will be repeatedly used in the proofs. Since the max function is monotonic, we have:
	\begin{align}
		\max\{[u_k+a]_+\} - \max\{[u_k]_+\} &\leq a, \label{subhomo}
	\end{align}
	for any $a \geq 0$. Further, $\max\{[au_k]_+\} = a\max\{[u_k]_+\}$ for any $a \geq 0$. Similarly, it can also be shown that
	\begin{align}
		\max\{[u_k+y_k]_+\} &\leq \max\{[u_k]_+\} + \max\{[y_k]_+\} \label{triangle2}.
	\end{align}
	
	For a $\mu$-strongly convex function $\varphi(\x)$, we have the quadratic lower bound 
	\begin{align}
		\varphi(\y) &\geq \varphi(\x) + \ip{\nabla \varphi (\x)}{\y - \x} + \frac{\mu}{2}\norm{\x-\y}^2, \label{qlb}
	\end{align}
	for all $\x,\y \in \text{dom } \varphi$. Hence, for $\x_\star = \arg\min_{\x} \varphi(\x)$, we have that
	which implies that 
	\begin{align}
		\varphi(\x_\star) + \frac{\mu}{2}\norm{\x-\x_\star}^2 \leq \varphi(\x), \label{sclb}
	\end{align}
	for all $\x \in \text{dom } \varphi$. If $\varphi$ is also $L$-smooth, we have the quadratic upper bound as well as the co-coercivity property:
	\begin{align}
		\varphi(\y) &\leq \varphi(\x) + \ip{\nabla \varphi (\x)}{\y - \x} + \frac{L}{2}\norm{\x-\y}^2 \label{qub}\\
		\varphi(\y) &\geq \varphi(\x) + \ip{\nabla \varphi (\x)}{\y - \x} + \tfrac{1}{2L}\norm{\nabla \varphi(\x)-\nabla \varphi(\y)}^2. \label{coco}
	\end{align}


We also list some of the common norm inequalities, which follow from the Cauchy-Schwarz inequality, Young's inequality, and the triangle inequality:
\begin{align}
	\ip{\u}{\v} &\leq \norm{\u}\norm{\v} \leq \tfrac{\varepsilon}{2}\norm{\u}^2+\tfrac{1}{2\varepsilon}\norm{\v}^2, \label{young}\\
	\norm{\u+\v} &\leq \norm{\u} + \norm{\v}, 	\label{triangle}
\end{align}
for $\varepsilon > 0$. Combining the two for $\varepsilon = 1$, we obtain
\begin{align}
	\norm{\u + \v}^2 \leq 2\norm{\u}^2 + 2\norm{\v}^2. \label{peterpaul}
\end{align} 

\section{Proof of Theorem \ref{prox-thm}}\label{pfprox-thm}
We begin by establishing a key lemma using the update equation, convexity, and the smoothness of the constraint functions $g_k$. 

\begin{lemma}\label{3plemma}
	Under Assumption \ref{a1}, the update \eqref{proxlin-x} implies that
	\begin{align}
		&\ip{\nabla f_{i_t}(\x_t)}{\x_{t+1} - \x_\star} + h(\x_{t+1}) + \gamma \max\{[g_k(\x_{t+1})]_+\} \nonumber\\
		& \leq h(\x_\star) + \tfrac{1}{2\eta_t}\norm{\x_t - \x_\star}^2  - \tfrac{1}{2\eta_t}\norm{\x_{t+1}-\x_\star}^2 \nonumber\\
		&\hspace{2cm}-\left(\tfrac{1}{2\eta_t}-\tfrac{\gamma L_g}{2}\right)\norm{\x_{t+1} - \x_t}^2. \label{3plemma-eq}
	\end{align}
\end{lemma}
\begin{IEEEproof}
	Since $\x_{t+1}$ is obtained by minimizing a $\frac{1}{\eta_t}$-strongly convex function in \eqref{proxlin-x}, we have from \eqref{sclb} that
	\begin{align}
		&\ip{\nabla f_{i_t}(\x_t)}{\x_{t+1} - \x_\star} + h(\x_{t+1}) + \tfrac{1}{2\eta_t}\norm{\x_{t+1} - \x_t}^2  \nonumber\\
		&\hspace{1cm}+ \gamma \max\{[g_k(\x_t) + \ip{\nabla g_k(\x_t)}{\x_{t+1} - \x_t}]_{+}\} \nonumber\\
		& \leq h(\x_\star) + \tfrac{1}{2\eta_t}\norm{\x_t - \x_\star}^2   - \tfrac{1}{2\eta_t}\norm{\x_{t+1}-\x_\star}^2 \nonumber\\
		&\hspace{1cm}+ \gamma \max\{[g_k(\x_t) + \ip{\nabla g_k(\x_t)}{\x_\star - \x_t}]_{+}\} \label{3p-1}\\
		&\leq h(\x_\star) + \tfrac{1}{2\eta_t}\norm{\x_t - \x_\star}^2   - \tfrac{1}{2\eta_t}\norm{\x_{t+1}-\x_\star}^2, \label{3p-1b}
	\end{align}
	where \eqref{3p-1b} follows from the convexity of $g_k$ and the feasibility of $\x_\star$, so that
	\begin{align}
		g_k(\x_t) + \ip{\nabla g_k(\x_t)}{\x_\star-\x_t} &\leq g_k(\x_\star) \leq 0 , \label{gconv} 
	\end{align}
	and hence, $\max\{[g_k(\x_t) + \ip{\nabla g_k(\x_t)}{\x_\star - \x_t}]_{+}\} = 0$. Likewise, since $g_k$ is $L_g$-smooth, we have from the \eqref{qub} and \eqref{subhomo}-\eqref{triangle2} that
	\begin{align}
		\max\{[g_k(\x_{t+1})]_+\} &\leq \max\{[g_k(\x_t) + \ip{\nabla g_k(\x_t)}{\x_{t+1}-\x_t}]_+\} \nonumber\\
		&+ \tfrac{L_g}{2}\norm{\x_t - \x_{t+1}}^2 . \label{3p-2}
	\end{align}
	Multiplying \eqref{3p-2} by $\gamma$ and substituting into \eqref{3p-1b}, we obtain the required result. 
\end{IEEEproof}

We are now ready to establish the one-step inequality for the SSQP algorithm. 

\begin{lemma}\label{step}
	Under Assumptions \eqref{a1} and \eqref{gradnoi}, we have for $\eta_t \leq \frac{1}{2(L_f+\max\{\gamma L_g,L_f\})}$, where $F(\x) = f(\x)+h(\x)+\gamma \max\{[g_k(\x)]_+\}$:
	\begin{align}
		\Et{F(\x_{t+1})}-F(\x_\star) &\leq \tfrac{1-\mu\eta_t(1-4\eta_t L_f)}{2\eta_t}\norm{\x_t-\x_\star}^2 \nonumber\\
		& - \tfrac{1}{2\eta_t}\Et{\norm{\x_{t+1}-\x_\star}^2} + 2\eta_t\sigma^2 .\label{step-eq}
	\end{align}
\end{lemma}
\begin{IEEEproof}
	Since $f$ is $L_f$-smooth, we have from \eqref{qub} that 
	\begin{align}
		f(\x_{t+1})&-f(\x_\star)\leq f(\x_t)  + \ip{\nabla f(\x_t)}{\x_{t+1}-\x_t} \nonumber\\
		&\hspace{1.4cm}+ \tfrac{L_f}{2}\norm{\x_{t+1}-\x_t}^2 - f(\x_\star) \label{fsmooth}\\
		&= -(f(\x_\star) - f(\x_t) - \ip{\nabla f(\x_t)}{\x_\star - \x_t}) \nonumber\\
		&+ \ip{\nabla f(\x_t)}{\x_{t+1}-\x_\star} + \tfrac{L_f}{2}\norm{\x_{t+1}-\x_t}^2.
	\end{align}
	Adding \eqref{3plemma-eq} and rearranging, we obtain
	\begin{align}
		F(\x_{t+1})&-F(\x_\star) \leq \tfrac{1}{2\eta_t}\norm{\x_t-\x_\star}^2 - \tfrac{1}{2\eta_t}\norm{\x_{t+1}-\x_\star}^2 \nonumber\\
		&- \tfrac{1-\eta_t(L_f+\gamma L_g)}{2\eta_t}\norm{\x_{t+1}-\x_t}^2 \nonumber\\
		&+ \ip{\nabla f(\x_t) - \nabla f_{i_t}(\x_t)}{\x_t-\x_\star} \nonumber\\
		&+ \ip{\nabla f(\x_t) - \nabla f_{i_t}(\x_t)}{\x_{t+1}-\x_t} \nonumber\\
		& -(f(\x_\star) - f(\x_t) - \ip{\nabla f(\x_t)}{\x_\star - \x_t}). \label{e0}
	\end{align}
	
	Since $i_t$ is selected at random, we have that
	\begin{align}
		\Et{\ip{\nabla f(\x_t) - \nabla f_{i_t}(\x_t)}{\x_t-\x_\star}} = 0 .\label{e1}
	\end{align}
	For the other term depending on $i_t$, we use \eqref{young} and \eqref{gradvar1} to obtain the bound
	\begin{align}
		&\Et{\ip{\nabla f(\x_t) - \nabla f_{i_t}(\x_t)}{\x_{t+1}-\x_t}} \nonumber\\
		&\lfrom{young} \eta_t\Et{\norm{\nabla f(\x_t) - \nabla f_{i_t}(\x_t)}^2} + \tfrac{1}{4\eta_t}\Et{\norm{\x_{t+1}-\x_t}^2} \nonumber\\
		&\lfrom{gradvar1} 4\eta_tL_f(f(\x_\star) - f(\x_t) - \ip{\nabla f(\x_t)}{\x_\star - \x_t}) + 2\eta_t\sigma^2 \nonumber\\
		&\hspace{2cm}+ \tfrac{1}{4\eta_t}\Et{\norm{\x_{t+1}-\x_t}^2}\label{e2}.
	\end{align}
	Therefore, taking expectation with respect to $i_t$ in \eqref{e0} and substituting \eqref{e1}-\eqref{e2}, we obtain
	\begin{align}
		&\Et{F(\x_{t+1})}-F(\x_\star) \leq \tfrac{1}{2\eta_t}\norm{\x_t-\x_\star}^2 + 2\eta_t\sigma^2 \nonumber\\
		&- \tfrac{1}{2\eta_t}\Et{\norm{\x_{t+1}-\x_\star}^2} - \tfrac{1-2\eta_t(L_f+\gamma L_g)}{4\eta_t}\Et{\norm{\x_{t+1}-\x_t}^2} \nonumber\\
		&-(1-4\eta_t L_f)(f(\x_\star) - f(\x_t) - \ip{\nabla f(\x_t)}{\x_\star - \x_t}). \label{prefinallem1}
	\end{align}
	The fourth term on the right is non-positive and can be dropped for $\eta_t \leq \frac{1}{2(L_f+\gamma L_g)}$. Finally, for $\eta_t \leq \frac{1}{4L_f}$, the last term on the right can be bounded from \eqref{qlb}, 
	\begin{align}
		-&(1-4\eta_t L_f)(f(\x_\star) - f(\x_t) - \ip{\nabla f(\x_t)}{\x_\star - \x_t}) \nonumber\\
		&\hspace{2cm}\leq -\tfrac{\mu(1-4\eta_t L_f)}{2}\norm{\x_t-\x_\star}^2,
	\end{align}
	which upon substituting into \eqref{prefinallem1}, yields the required result. Further, both requirements for $\eta_t$ are satisfied when $\eta_t \leq \frac{1}{2(L_f+\max\{L_f,\gamma L_g\})}$.
\end{IEEEproof}
Having established the one-step inequality, we can now obtain the required oracle complexities. 
\begin{IEEEproof}[Proof of Theorem \ref{prox-thm}]
	For the sake of brevity, recall the definition of $\Delta_t$ from Sec. \ref{penaltyequivalent} and also define $\delta_t := \EE\norm{\x_t - \x_\star}^2$. Taking full expectation in \eqref{step-eq} and using the definitions of $\Delta_t$ and $\delta_t$, we obtain:
	\begin{align}
		\Delta_{t+1} \leq \tfrac{1-\mu\eta_t(1-4\eta_t L_f)}{2\eta_t}\delta_t - \tfrac{1}{2\eta_t}\delta_{t+1} + 2\eta_t\sigma^2. 
	\end{align}
	Let us consider the convex case first, where we set $\mu = 0$, yielding $\Delta_{t+1} \leq \tfrac{1}{2\eta_t}\delta_t - \tfrac{1}{2\eta_t}\delta_{t+1} + 2\eta_t\sigma^2$. Multiplying by $\eta_t$ on both sides for $\eta_{t+1} \leq \eta_t$ and taking sum for $t = 0, \ldots, T-1$, we obtain:
	\begin{align}
		\sum_{t=1}^T \eta_t\Delta_t &\leq \tfrac{\delta_0 - \delta_{T}}{2} + 2\sigma^2\sum_{t=0}^{T-1} \eta_t^2 \leq \tfrac{\delta_0}{2}
		+ 2\sigma^2\sum_{t=0}^{T-1} \eta_t^2.
	\end{align}
	Therefore, from the convexity of $F$, we have the following bound for the averaged iterate $\bar{\x}_T := \left(\sum_{t=1}^T \eta_t\x_t\right)/\left(\sum_{t=1}^T\eta_t\right)$:
	\begin{align*}
		\E{F(\bar{\x}_T)}-F(\x_\star) \leq \tfrac{\sum_{t=1}^T \eta_t\Delta_t}{\sum_{t=1}^T\eta_t} \leq \tfrac{\delta_0 +  4\sigma^2\sum_{t=0}^{T-1}\eta_t^2}{2\sum_{t=1}^T \eta_t},
	\end{align*}
	With the stepsize rule $\eta_t = \frac{\eta_0}{\sqrt{t+1}}$ where $\eta_0 \leq \frac{1}{4L}$, we obtain
	\begin{align}\label{rate_ineq_convex_fixT}
		\E{F(\bar{\x}_T)}-F(\x_\star) \leq \tfrac{\delta_0 + 4\sigma^2\eta_0^2(1+\log(T))}{2\eta_0\sqrt{T}}. 
	\end{align}
	Alternatively, setting $\eta_t = \eta_0/\sqrt{T}$ for $\eta_0 = \min\{\frac{\sqrt{\delta_0}}{2\sigma},\frac{1}{4L}\}$, the bound becomes
	\begin{align}\label{rate_ineq_convex}
		\E{F(\bar{\x}_T)}-F(\x_\star) &\leq \tfrac{2}{\sqrt{T}}\max\{2L\delta_0, \sigma\sqrt{\delta_0}\}.
	\end{align}
	Since $\max\{[g_k(\x_\star)]_+\} = 0$ and $\E{\max\{[g_k(\bar{\x}_T)]_+\}} \geq 0$, we can drop these terms  from the left of \eqref{rate_ineq_convex_fixT}-\eqref{rate_ineq_convex} to yield the desired bounds on the optimality gap. 
	
	To bound the constraint violation, let $w_T := \max_k\{g_k(\bar{\x}_T)\}$. Taking expectation in \eqref{optimalKKT} and rearranging, we obtain
	\begin{align}\label{con_ineq_1}
		f(\x_\star)+h(\x_\star)-\E{f(\bar{\x}_T)+h(\bar{\x}_T)} &\leq\norm{\lmda_\star}_1 \E{w_T}\nonumber\\
		&\leq \tfrac{\bt}{\nu} \E{w_T}. 
	\end{align}
	Adding \eqref{rate_ineq_convex} and \eqref{con_ineq_1} we obtain
	\begin{align}\label{con_ineq_2}
		\left(\gamma - \tfrac{\bt}{\nu}\right) \E{w_T} \leq \tfrac{2}{\sqrt{T}}\max\{2L\delta_0, \sigma\sqrt{\delta_0}\},
	\end{align}
	which yields the desired bound and hence the SFO/QMO complexity for the convex case. 	
	
	For the strongly convex case, we first use \eqref{qlb} to write 
	\begin{align}
		F(\x_{t+1})-F(\x_\star) \geq \tfrac{\mu}{2}\norm{\x_{t+1}-\x_\star}^2 ,
	\end{align}
	so as to obtain
	\begin{align}
		\delta_{t+1}\leq \tfrac{1-\mu\eta_t(1-4\eta_tL_f)}{1+\mu\eta_t} \delta_t + \tfrac{4\eta_t^2\sigma^2}{1+\mu\eta_t} .
	\end{align}
	Suppose we choose $\eta_t \leq \frac{1}{8L} \leq \frac{1}{8L_f}$ so that $1-4\eta_tL_f \geq 1/2$ and the recursion becomes
	\begin{align}
		\delta_{t+1}\leq \tfrac{1-\mu\eta_t/2}{1+\mu\eta_t} \delta_t + \tfrac{4\eta_t^2\sigma^2}{1+\mu\eta_t}.
	\end{align}
	Thus if we set $\eta_t = \frac{2}{\mu(t+\omega+1)}$ and choose $\omega = \lfloor \frac{16L}{\mu}\rfloor$, it would follow that $\eta_t \leq \eta_0 \leq \frac{1}{8L}$ and the recursion can be written as
	\begin{align}\label{seq1}
		\delta_{t+1} \leq \tfrac{t+\omega}{t+\omega+3}\delta_t + \tfrac{16\sigma^2}{\mu^2(t+\omega+3)(t+\omega+1)} .
	\end{align} 
	Multiplying both sides by $(t+\omega+1)(t+\omega+2)(t+\omega+3)$ and summing telescopically, we obtain
	\begin{align}
		(T+\omega+2)&(T+\omega+1)(T+\omega)\delta_T \\
		\leq \omega(\omega+1)&(\omega+2)\delta_0 + \tfrac{8\sigma^2}{\mu^2}(T+\omega+1)(T+\omega+2). \nonumber
	\end{align}
	Rearranging and bounding the terms on the right, we obtain
	\begin{align}
		\delta_T &\leq \left(\tfrac{\omega+2}{T+\omega+2}\right)^3\delta_0 + \tfrac{8\sigma^2}{\mu^2(T+\omega)} \\
		&\leq \tfrac{(16\kappa+2)^3}{T^3}\delta_0 + \tfrac{8\sigma^2}{\mu^2T},\label{sc-sgd}
	\end{align}
	where $\kappa = L/\mu$. The bound in \eqref{sc-sgd} translates to an SFO complexity of $\O{\frac{\sigma^2}{\mu^2\epsilon}+\frac{\kappa \delta_0^{1/3}}{\epsilon^{1/3}}}$.
\end{IEEEproof}

\section{Proof of Theorem \ref{ssqp-skip-thm}}\label{pf-ssqp-skip-thm}
Before establishing the main result, we derive some intermediate results through the following lemmas. The following result is a consequence of the update in \eqref{ssqp-skip1} and the convexity and smoothness of $g_k$.
\begin{lemma}\label{prelim-skip}
	Under Assumption \ref{a1} and for $\y_\star = \nabla f(\x_\star)$,  it holds that
	\begin{align}\label{3pssqp-skip}
		p_t\norm{\hx_{t+1}-\x_\star}^2 \leq &p_t\norm{\tx_{t+1}-\x_\star}^2 - 2\eta_t \ip{\hx_{t+1}-\x_\star}{\y_t-\y_\star}\nonumber\\
		& - \br{p_t - \gamma\eta_tL_g}\norm{\hx_{t+1}-\tx_{t+1}}^2. 
	\end{align}
\end{lemma}
\begin{IEEEproof}
	We begin with establishing three point inequality for the update $\hx_{t+1}$. Since \eqref{ssqp-skip1} involves minimizing a $\frac{p_t}{\eta_t}$-strongly convex function, we have that
	\begin{align}\label{3p-skip}
		&\tfrac{p_t}{2\eta_t}\norm{\tx_{t+1}-\hx_{t+1}}^2 + \ip{\hx_{t+1}-\x_\star}{\y_t} + h(\hx_{t+1}) \nonumber\\
		&\hspace{5mm} + \gamma\max\bc{\bsq{g_k(\tx_{t+1}) + \ip{\nabla g_k(\tx_{t+1})}{\hx_{t+1}-\tx_{t+1}}}_+} \nonumber\\
		&\lfrom{sclb} -\tfrac{p_t}{2\eta_t}\norm{\hx_{t+1}-\x_\star}^2 + \tfrac{p_t}{2\eta_t}\norm{\tx_{t+1}-\x_\star}^2 + h(\x_\star) \nonumber\\
		&\hspace{-1mm}+ \gamma\max\bc{\bsq{g_k(\tx_{t+1}) + \ip{\nabla g_k(\tx_{t+1})}{\x_{\star}-\tx_{t+1}}}_+}. 
	\end{align}
	Next, the convexity and $L_g$-smoothness of $g_k$ allow us to similarly use \eqref{gconv} and \eqref{3p-2}, respectively, with $\x_{t+1}$ in place of $\hx_{t+1}$ and $\x_t$ in place of $\tx_{t+1}$, yielding
	\begin{align}\label{3p-mid}
		&\tfrac{p_t}{2\eta_t}\norm{\hx_{t+1}-\x_\star}^2 \leq \tfrac{p_t}{2\eta_t}\norm{\tx_{t+1}-\x_\star}^2 -  \ip{\hx_{t+1}-\x_\star}{\y_t} \nonumber\\
		&- \br{\tfrac{p_t}{2\eta_t} - \tfrac{\gamma L_g}{2}}\norm{\hx_{t+1}-\tx_{t+1}}^2 \nonumber\\
		& + h(\x_\star) - h(\hx_{t+1}) - \gamma\max\bc{\bsq{g_k(\hx_{t+1})}_+}.
	\end{align}
	Finally, the first order optimality condition of \eqref{penalty1} yields
	\begin{align}
		\x_\star &= \arg\min_{\u\in\Rn^d} \tfrac{1}{2\eta_t}\norm{\u-\x_\star}^2 + \ip{\u}{\y_\star} + h(\u)  \nonumber\\
		&+ \gamma\max\bc{\bsq{g_k(\x_\star) + \ip{\nabla g_k(\x_\star)}{\u-\x_\star}}_+}, \label{optcon1}
	\end{align}
	for $\y_\star = \nabla f(\x_\star)$. Since \eqref{optcon1} again involves minimization of a $\tfrac{1}{\eta_t}$-strongly convex function, we have from \eqref{sclb} that:
	\begin{align}
		&h(\x_\star) - h(\hx_{t+1}) + \ip{\x_\star-\hx_{t+1}}{\y_\star}  \nonumber\\
		&\leq \gamma\max\bc{\bsq{g_k(\x_\star) + \ip{\nabla g_k(\x_\star)}{\hx_{t+1}-\x_\star}}_+} \\
		&\leq \gamma\max\bc{\bsq{g_k(\hx_{t+1})}_+}, \label{optskip}
	\end{align}
	where \eqref{optskip} follows from the convexity of $g_k$ and monotonicity of the $\max\{[\cdot]_+\}$ operator. Substituting \eqref{optskip} in \eqref{3p-mid} and multiplying by $2\eta_t$, we obtain the required result.
\end{IEEEproof}

The next lemma obtains a recursive inequality incorporating the effect of skipping \eqref{ssqp-skip1}. Recall that in Algorithm \ref{ssqp-skip}, $w_t \sim $Bernoulli$(p_t)$, and let $\Ewt{\cdot}$ denote the expectation with respect to $w_t$. 
\begin{lemma}\label{interskip}
	For $\eta_t \leq \frac{p_t}{2\gamma L_g}$ and $\y_\star = \nabla f(\x_\star)$, the updates in Algorithm \ref{ssqp-skip} imply that
	\begin{align}\label{onestprec}
		&\Ewt{\norm{\x_{t+1} - \x_\star}^2} + \tfrac{2\eta_t^2}{p_t^2}\Ewt{\norm{\y_{t+1} - \y_\star}^2}  \\
		&\leq \norm{\x_t-\x_\star - \eta_t\br{\nabla f_{i_t}(\x_t)-\y_\star}}^2 + \tfrac{\eta_t^2(2-p_t^2)}{p_t^2}\norm{\y_t-\y_\star}^2. \nonumber
	\end{align} 
\end{lemma}
\begin{IEEEproof}
	From the update in \eqref{ssqp-up} and the result of Lemma \ref{prelim-skip}, we have that
	\begin{align}
		\Ewt{\norm{\x_{t+1} - \x_\star}^2} &= p_t\norm{\hx_{t+1} - \x_\star}^2 + (1-p_t)\norm{\tx_{t+1} - \x_\star}^2 \nonumber\\
		&\hspace{-3mm}\lfrom{3pssqp-skip} \norm{\tx_{t+1} - \x_\star}^2 -  2\eta_t\ip{\hx_{t+1}-\x_\star}{\y_t-\y_\star} \nonumber\\
		&- \br{p_t - \eta_t\gamma L_g}\norm{\hx_{t+1}-\tx_{t+1}}^2. \label{proofskip1}
	\end{align}
	We also note from the update of $\y_{t+1}$ in Algorithm \ref{ssqp-skip} that $\y_{t+1} = \y_t + \tfrac{w_tp_t}{2\eta_t}(\hx_{t+1}-\tx_{t+1})$ so that 
	\begin{align}
		\Ewt{\norm{\y_{t+1}-\y_\star}^2} &= p_t\norm{\y_t-\y_\star + \tfrac{p_t}{2\eta_t}\br{\hx_{t+1}-\tx_{t+1}}}^2 \nonumber\\
		&+ (1-p_t)\norm{\y_t-\y_\star}^2 \\
		&= \norm{\y_t-\y_\star}^2 + \tfrac{p_t^3}{4\eta_t^2}\norm{\hx_{t+1}-\tx_{t+1}}^2 \nonumber\\
		&+ \tfrac{p_t^2}{\eta_t}\ip{\hx_{t+1}-\tx_{t+1}}{\y_t-\y_\star}. \label{proofskip2}
	\end{align}
	Multiplying \eqref{proofskip2} by $\tfrac{2\eta_t^2}{p_t^2}$ and adding with \eqref{proofskip1}, we obtain
	\begin{align}
		&\Ewt{\norm{\x_{t+1} - \x_\star}^2} + \tfrac{2\eta_t^2}{p_t^2}	\Ewt{\norm{\y_{t+1}-\y_\star}^2} \nonumber\\
		&\leq \norm{\tx_{t+1}-\x_\star}^2 -2\eta_t\ip{\tx_{t+1}-\x_\star}{\y_t-\y_\star} \nonumber\\
		& + \tfrac{2\eta_t^2}{p_t^2}\norm{\y_t-\y_\star}^2 - \br{\tfrac{p_t}{2}-\eta_t\gamma L_g}\norm{\hx_{t+1}-\tx_{t+1}}^2, \label{rec-skip}
	\end{align}
	where the last term can be dropped if $p_t > 2\eta_t\gamma L_g$. 
	
	Next, we have from \eqref{sgd-up} that
	\begin{align}\label{tx-xs}
		&\norm{\tx_{t+1}-\x_\star}^2 = \norm{\x_t-\x_\star - \eta_t\nabla f_{i_t}(\x_t) + \eta_t \y_t}^2 \nonumber\\
		&= \norm{\x_t-\x_\star - \eta_t\br{\nabla f_{i_t}(\x_t)-\y_\star} + \eta_t\br{\y_t - \y_\star}}^2 \nonumber\\
		&= \norm{\x_t-\x_\star - \eta_t\br{\nabla f_{i_t}(\x_t)-\y_\star}}^2 + \eta_t^2\norm{\y_t - \y_\star}^2 \nonumber\\
		&+ 2\eta_t\ip{\x_t-\x_\star - \eta_t\br{\nabla f_{i_t}(\x_t)-\y_\star}}{\y_t - \y_\star},
	\end{align}
	and similarly
	\begin{align}\label{yttxip}
		&\ip{\tx_{t+1}-\x_\star}{\y_t-\y_\star} \\
		&= \ip{\x_t-\x_\star - \eta_t\nabla f_{i_t}(\x_t) + \eta_t \y_t}{\y_t-\y_\star}\nonumber\\
		&= \ip{\x_t-\x_\star - \eta_t\br{\nabla f_{i_t}(\x_t)-\y_\star} + \eta_t\br{\y_t - \y_\star}}{\y_t-\y_\star} \nonumber\\
		& = \eta_t\norm{\y_t-\y_\star}^2 + \ip{\x_t-\x_\star - \eta_t\br{\nabla f_{i_t}(\x_t)-\y_\star}}{\y_t-\y_\star}. \nonumber
	\end{align}
	Substituting \eqref{tx-xs}-\eqref{yttxip} in \eqref{rec-skip}, we obtain the required result. 
\end{IEEEproof}
Having derived the key recursive inequality, we now proceed with proving the main result. 
\begin{IEEEproof}[Proof of Theorem \ref{ssqp-skip-thm}]
	From the update in \eqref{sgd-up}, we obtain
	\begin{align}
		&\norm{\x_t-\x_\star - \eta_t\br{\nabla f_{i_t}(\x_t)-\y_\star}}^2 = \norm{\x_t-\x_\star}^2 \\
		&+ \eta_t^2\norm{\nabla f_{i_t}(\x_t)-\y_\star}^2 - 2\eta_t\ip{\x_t-\x_\star}{\nabla f_{i_t}(\x_t)-\y_\star}. \nonumber
	\end{align}
	Taking expectation with respect to the random variable $i_t$, we obtain
	\begin{align}\label{xtxsrel}
		&\Et{\norm{\x_t-\x_\star - \eta_t\br{\nabla f_{i_t}(\x_t)-\y_\star}}^2} =\norm{\x_t-\x_\star}^2 \\
		&+ \eta_t^2\Et{\norm{\nabla f_{i_t}(\x_t)-\y_\star}^2} - 2\eta_t\ip{\x_t-\x_\star}{\nabla f(\x_t)-\y_\star}\nonumber.
	\end{align}
	Recalling that $\y_\star = \nabla f(\x_\star)$ and using \eqref{gradvar1}, the second term on the right can be bounded as
	\begin{align}
		&\Et{\norm{\nabla f_{i_t}(\x_t)-\y_\star}^2} \leq4L_fD_f(\x_\star,\x_t) + 2\sigma^2.\label{proofskip3}
	\end{align}
	We further have that 
	\begin{align}
		\ip{\x_t-\x_\star}{\nabla f(\x_t)-\y_\star} = D_f(\x_\star,\x_t) + D_f(\x_t,\x_\star).\label{proofskip4}
	\end{align}
	Substituting \eqref{proofskip3}-\eqref{proofskip4} into \eqref{xtxsrel}, we obtain
	\begin{align}
		&\Et{\norm{\x_t-\x_\star - \eta_t\br{\nabla f_{i_t}(\x_t)-\y_\star}}^2} \leq \norm{\x_t-\x_\star}^2\\
		& - 2\eta_t(1-2\eta_tL_f)D_f(\x_\star,\x_t) - 2\eta_tD_f(\x_t,\x_\star) + 2\eta_t^2\sigma^2. \nonumber
	\end{align}
	Here, we can drop the non-positive second term on the right since $\eta_t \leq \frac{p_t}{2L} \leq \frac{1}{2L_f}$ and use \eqref{qlb} to obtain
	\begin{align}
		&\Et{\norm{\x_t-\x_\star - \eta_t\br{\nabla f_{i_t}(\x_t)-\y_\star}}^2} \nonumber\\
		&\leq (1-\mu\eta_t)\norm{\x_t-\x_\star}^2 + 2\eta_t^2\sigma^2 .\label{proofskip5}
	\end{align}
	Taking full expectation in \eqref{onestprec} and substituting \eqref{proofskip5}, we obtain 
	\begin{align}
		&\E{\norm{\x_{t+1} - \x_\star}^2} + \tfrac{2\eta_t^2}{p_t^2}\E{\norm{\y_{t+1} - \y_\star}^2}  \nonumber\\
		&\leq (1-\mu\eta_t) \E{\norm{\x_t-\x_\star}^2} +2\eta_t^2\sigma^2 \nonumber\\
		&+\tfrac{\eta_t^2(2-p_t^2)}{p_t^2}\E{\norm{\y_t-\y_\star}^2}. 
	\end{align}
	Let us denote $\Psi_{t+1} := \E{\norm{\x_{t+1} - \x_\star}^2 + \frac{\eta_t}{\mu}\norm{\y_{t+1} - \y_\star}^2}$. From the initialization and Assumption \ref{gradnoi}, we have that $\EE\norm{\y_0 - \y_\star}^2 \leq 2L_f^2\norm{\x_0 - \x_\star}^2 + 4\sigma^2$, implying that $\Psi_0 \leq \delta_0 + \frac{\eta_{-1}}{\mu}\EE\norm{\y_0-\y_\star}^2 \leq \br{1+4\kappa^2}\delta_0 + \tfrac{4\sigma^2}{\mu^2}$ for $\eta_{-1} \leq \frac{2}{\mu}$. If we set $p_t = \sqrt{2\mu\eta_t}$, we would obtain the one-step inequality:
	\begin{align}
		\Psi_{t+1} \leq (1-\mu\eta_t)\Psi_t + 2\eta_t^2\sigma^2,
	\end{align}
	for $\eta_t \leq \eta_{t-1}$. Recall that Lemma \ref{interskip} also requires that $\eta_t \leq p_t/2L$ or equivalently $\eta_t \leq \frac{\mu}{2L^2}$. Therefore, if we set $\eta_t = \frac{2}{\mu(t+\omega+1)}$ for $t\geq -1$ with $\omega = \lceil \frac{4L^2}{\mu^2} \rceil$, and proceed as in \eqref{seq1}-\eqref{sc-sgd}, we obtain the bound
	\begin{align}
		\Psi_T  \leq \tfrac{8\sigma^2}{\mu^2T} + \tfrac{\omega^2\Psi_0}{T^2} \leq \tfrac{8\sigma^2}{\mu^2T} + \tfrac{4\kappa^4((1+4\kappa^2)\mu^2\delta_0+4\sigma^2)}{\mu^2T^2}.
	\end{align}
	The obtained bounds hence translate to an SFO complexity of $\O{\frac{\sigma^2}{\mu^2\epsilon} + \kappa^2\tfrac{\kappa\sqrt{\delta_0}+\sigma}{\sqrt{\epsilon}}}$. However, the average number of calls to the QMO are bounded as $\sum_{t=1}^{T-1} p_t \leq 4\sqrt{T+\omega}$ or $\O{\frac{\sigma}{\kappa + \mu\sqrt{\epsilon}} + \frac{\kappa\sqrt{\kappa\sqrt{\delta_0}+\sigma}}{\epsilon^{1/4}}}$. 
\end{IEEEproof}

\footnotesize

\bibliographystyle{IEEEtran} 
\bibliography{IEEEabrv,references}

\begin{thebibliography}{10}
\providecommand{\url}[1]{#1}
\csname url@samestyle\endcsname
\providecommand{\newblock}{\relax}
\providecommand{\bibinfo}[2]{#2}
\providecommand{\BIBentrySTDinterwordspacing}{\spaceskip=0pt\relax}
\providecommand{\BIBentryALTinterwordstretchfactor}{4}
\providecommand{\BIBentryALTinterwordspacing}{\spaceskip=\fontdimen2\font plus
\BIBentryALTinterwordstretchfactor\fontdimen3\font minus
  \fontdimen4\font\relax}
\providecommand{\BIBforeignlanguage}[2]{{%
\expandafter\ifx\csname l@#1\endcsname\relax
\typeout{** WARNING: IEEEtran.bst: No hyphenation pattern has been}%
\typeout{** loaded for the language `#1'. Using the pattern for}%
\typeout{** the default language instead.}%
\else
\language=\csname l@#1\endcsname
\fi
#2}}
\providecommand{\BIBdecl}{\relax}
\BIBdecl

\bibitem{dieuleveut2023stochastic}
A.~Dieuleveut, G.~Fort, E.~Moulines, and H.-T. Wai, ``Stochastic approximation
  beyond gradient for signal processing and machine learning,'' \emph{IEEE
  Transactions on Signal Processing}, vol.~71, pp. 3117--3148, 2023.

\bibitem{ribeiro2010ergodic}
A.~Ribeiro, ``Ergodic stochastic optimization algorithms for wireless
  communication and networking,'' \emph{{IEEE} Trans. Signal Process.},
  vol.~58, no.~12, pp. 6369--6386, 2010.

\bibitem{xin2020variance}
R.~Xin, U.~A. Khan, and S.~Kar, ``Variance-reduced decentralized stochastic
  optimization with accelerated convergence,'' \emph{IEEE Transactions on
  Signal Processing}, vol.~68, pp. 6255--6271, 2020.

\bibitem{yooEnsemble}
C.~Yoo, J.~J. Heon~Lee, S.~Anstee, and R.~Fitch, ``Path planning in uncertain
  ocean currents using ensemble forecasts,'' in \emph{IEEE Intl. Conf. on
  Robotics and Automation (ICRA)}, 2021, pp. 8323--8329.

\bibitem{gollamudi1998set}
S.~Gollamudi, S.~Nagaraj, S.~Kapoor, and Y.-F. Huang, ``Set-membership
  filtering and a set-membership normalized lms algorithm with an adaptive step
  size,'' \emph{IEEE Signal Processing Letters}, vol.~5, no.~5, pp. 111--114,
  1998.

\bibitem{bhotto2011robust}
M.~Z.~A. Bhotto and A.~Antoniou, ``Robust set-membership affine-projection
  adaptive-filtering algorithm,'' \emph{IEEE Transactions on Signal
  Processing}, vol.~60, no.~1, pp. 73--81, 2011.

\bibitem{flores2019set}
A.~Flores and R.~C. de~Lamare, ``Set-membership adaptive kernel nlms
  algorithms: Design and analysis,'' \emph{Signal Processing}, vol. 154, 2019.

\bibitem{necoara2022stochastic}
I.~Necoara and N.~K. Singh, ``Stochastic subgradient for composite convex
  optimization with functional constraints,'' \emph{Journal of Machine Learning
  Research}, vol.~23, no. 265, pp. 1--35, 2022.

\bibitem{lan2020algorithms}
G.~Lan and Z.~Zhou, ``Algorithms for stochastic optimization with function or
  expectation constraints,'' \emph{Computational Optimization and
  Applications}, vol.~76, no.~2, pp. 461--498, 2020.

\bibitem{nedic2019random}
A.~Nedi{\'c} and I.~Necoara, ``Random minibatch subgradient algorithms for
  convex problems with functional constraints,'' \emph{Applied Mathematics and
  Optimization}, vol.~80, no.~3, pp. 801--833, 2019.

\bibitem{basu2019optimal}
K.~Basu and P.~Nandy, ``Optimal convergence for stochastic optimization with
  multiple expectation constraints,'' \emph{arXiv preprint arXiv:1906.03401},
  2019.

\bibitem{bayandina2018mirror}
A.~Bayandina, P.~Dvurechensky, A.~Gasnikov, F.~Stonyakin, and A.~Titov,
  ``Mirror descent and convex optimization problems with non-smooth inequality
  constraints,'' in \emph{Large-Scale and Distributed Optimization}.\hskip 1em
  plus 0.5em minus 0.4em\relax Springer, 2018, pp. 181--213.

\bibitem{stonyakin2019adaptive}
F.~S. Stonyakin, M.~Alkousa, A.~N. Stepanov, and A.~A. Titov, ``Adaptive mirror
  descent algorithms for convex and strongly convex optimization problems with
  functional constraints,'' \emph{Journal of Applied and Industrial
  Mathematics}, vol.~13, no.~3, pp. 557--574, 2019.

\bibitem{alkousa2020modification}
M.~S. Alkousa, ``On modification of an adaptive stochastic mirror descent
  algorithm for convex optimization problems with functional constraints,'' in
  \emph{Computational Mathematics and Applications}.\hskip 1em plus 0.5em minus
  0.4em\relax Springer, 2020, pp. 47--63.

\bibitem{alkousa2019some}
M.~Alkousa \emph{et~al.}, ``On some stochastic mirror descent methods for
  constrained online optimization problems,'' \emph{Computer research and
  modeling}, vol.~11, no.~2, pp. 205--217, 2019.

\bibitem{nesterov2009primal}
Y.~Nesterov, ``Primal-dual subgradient methods for convex problems,''
  \emph{Mathematical programming}, vol. 120, no.~1, pp. 221--259, 2009.

\bibitem{nemirovski2009robust}
A.~Nemirovski, A.~Juditsky, G.~Lan, and A.~Shapiro, ``Robust stochastic
  approximation approach to stochastic programming,'' \emph{SIAM Journal on
  optimization}, vol.~19, no.~4, pp. 1574--1609, 2009.

\bibitem{xu2020primal}
Y.~Xu, ``Primal-dual stochastic gradient method for convex programs with many
  functional constraints,'' \emph{SIAM Journal on Optimization}, vol.~30,
  no.~2, pp. 1664--1692, 2020.

\bibitem{yazdandoost2019randomized}
E.~Yazdandoost~Hamedani, A.~Jalilzadeh, and N.~Serhat~Aybat, ``A randomized
  block-coordinate primal-dual method for large-scale stochastic saddle point
  problems,'' \emph{arXiv e-prints}, pp. arXiv--1907, 2019.

\bibitem{madavan2021stochastic}
A.~N. Madavan and S.~Bose, ``A stochastic primal-dual method for optimization
  with conditional value at risk constraints,'' \emph{Journal of Optimization
  Theory and Applications}, vol. 190, no.~2, pp. 428--460, 2021.

\bibitem{yan2022adaptive}
Y.~Yan and Y.~Xu, ``Adaptive primal-dual stochastic gradient method for
  expectation-constrained convex stochastic programs,'' \emph{Mathematical
  Programming Computation}, vol.~14, no.~2, pp. 319--363, 2022.

\bibitem{yuan2018online}
J.~Yuan and A.~Lamperski, ``Online convex optimization for cumulative
  constraints,'' in \emph{Proc. of the Intl. Conf. on Neural Information
  Processing Systems}, 2018, p. 6140–6149.

\bibitem{yu2017online}
H.~Yu and M.~J. Neely, ``Online convex optimization with stochastic
  constraints,'' in \emph{Proc. of the Intl. Conf. on Neural Information
  Processing Systems}, Long Beach, CA, USA, 2017.

\bibitem{lin2018level}
Q.~Lin, R.~Ma, and T.~Yang, ``Level-set methods for finite-sum constrained
  convex optimization,'' in \emph{Intl. Conf. on Machine Learning}.\hskip 1em
  plus 0.5em minus 0.4em\relax PMLR, 2018, pp. 3112--3121.

\bibitem{han1979exact}
S.-P. Han and O.~L. Mangasarian, ``Exact penalty functions in nonlinear
  programming,'' \emph{Mathematical programming}, vol.~17, no.~1, pp. 251--269,
  1979.

\bibitem{lin2003some}
G.-H. Lin and M.~Fukushima, ``Some exact penalty results for nonlinear programs
  and mathematical programs with equilibrium constraints,'' \emph{Journal of
  Optimization Theory and Applications}, vol. 118, no.~1, pp. 67--80, 2003.

\bibitem{zhang2021stochastic}
J.~Zhang and L.~Xiao, ``Stochastic variance-reduced prox-linear algorithms for
  nonconvex composite optimization,'' \emph{Mathematical Programming}, pp.
  1--43, 2021.

\bibitem{bertsekas}
D.~P. Bertsekas, \emph{Nonlinear Programming}.\hskip 1em plus 0.5em minus
  0.4em\relax Athena Scientific, 1997.

\bibitem{bedi2018tracking}
A.~S. Bedi, P.~Sarma, and K.~Rajawat, ``Tracking moving agents via inexact
  online gradient descent algorithm,'' \emph{IEEE Journal of Selected Topics in
  Signal Processing}, vol.~12, no.~1, pp. 202--217, 2018.

\bibitem{gill2012sequential}
P.~E. Gill and E.~Wong, ``Sequential quadratic programming methods,'' in
  \emph{Mixed integer nonlinear programming}.\hskip 1em plus 0.5em minus
  0.4em\relax Springer, 2012, pp. 147--224.

\bibitem{curtis2024sequential}
F.~E. Curtis, D.~P. Robinson, and B.~Zhou, ``Sequential quadratic optimization
  for stochastic optimization with deterministic nonlinear inequality and
  equality constraints,'' \emph{SIAM Journal on Optimization}, vol.~34, no.~4,
  pp. 3592--3622, 2024.

\bibitem{doikov2022optimization}
N.~Doikov and Y.~Nesterov, ``High-order optimization methods for fully
  composite problems,'' \emph{SIAM Journal on Optimization}, vol.~32, no.~3,
  pp. 2402--2427, 2022.

\bibitem{davis2019stochastic}
D.~Davis and D.~Drusvyatskiy, ``Stochastic model-based minimization of weakly
  convex functions,'' \emph{SIAM Journal on Optimization}, vol.~29, no.~1, pp.
  207--239, 2019.

\bibitem{wang2017penalty}
X.~Wang, S.~Ma, and Y.-x. Yuan, ``Penalty methods with stochastic approximation
  for stochastic nonlinear programming,'' \emph{Mathematics of computation},
  vol.~86, no. 306, pp. 1793--1820, 2017.

\bibitem{xiao2019penalized}
X.~Xiao, ``Penalized stochastic gradient methods for stochastic convex
  optimization with expectation constraints,'' \emph{Optimization-online},
  2019.

\bibitem{thomdapu2019optimal}
S.~T. Thomdapu and K.~Rajawat, ``Optimal design of queuing systems via
  compositional stochastic programming,'' \emph{IEEE Trans. Commun.}, vol.~67,
  no.~12, pp. 8460--8474, 2019.

\bibitem{bertsekas2014constrained}
D.~P. Bertsekas, \emph{Constrained optimization and Lagrange multiplier
  methods}.\hskip 1em plus 0.5em minus 0.4em\relax Academic press, 2014.

\bibitem{zhang2013logt}
L.~Zhang, T.~Yang, R.~Jin, and X.~He, ``$\mathcal{O}(\log t)$ projections for
  stochastic optimization of smooth and strongly convex functions,'' in
  \emph{Intl Conf. on Machine Learning}.\hskip 1em plus 0.5em minus 0.4em\relax
  PMLR, 2013, pp. 1121--1129.

\bibitem{chen2016optimal}
J.~Chen, T.~Yang, Q.~Lin, L.~Zhang, and Y.~Chang, ``Optimal stochastic strongly
  convex optimization with a logarithmic number of projections,'' in
  \emph{Thirty-Second Conf. on Uncertainty in Artificial Intelligence}.\hskip
  1em plus 0.5em minus 0.4em\relax AUAI Press, 2016, pp. 122--131.

\bibitem{mahdavi2012stochastic}
M.~Mahdavi, T.~Yang, R.~Jin, S.~Zhu, and J.~Yi, ``Stochastic gradient descent
  with only one projection,'' in \emph{Proc. of the Intl. Conf. on Neural
  Information Processing Systems}, 2012, pp. 494--502.

\bibitem{lan2019unified}
G.~Lan, Z.~Li, and Y.~Zhou, ``A unified variance-reduced accelerated gradient
  method for convex optimization,'' in \emph{Proc. of the Intl. Conf. on Neural
  Information Processing Systems}, 2019, pp. 10\,462--10\,472.

\bibitem{jalilzadeh2021primal}
A.~Jalilzadeh, ``Primal-dual incremental gradient method for nonsmooth and
  convex optimization problems,'' \emph{Optimization Letters}, vol.~15, no.~8,
  pp. 2541--2554, 2021.

\bibitem{fercoq2019almost}
O.~Fercoq, A.~Alacaoglu, I.~Necoara, and V.~Cevher, ``Almost surely constrained
  convex optimization,'' in \emph{International Conf. on Machine
  Learning}.\hskip 1em plus 0.5em minus 0.4em\relax PMLR, 2019, pp. 1910--1919.

\bibitem{kundu2018convex}
A.~Kundu, F.~Bach, and C.~Bhattacharya, ``Convex optimization over intersection
  of simple sets: improved convergence rate guarantees via an exact penalty
  approach,'' in \emph{International Conf. on Artificial Intelligence and
  Statistics}.\hskip 1em plus 0.5em minus 0.4em\relax PMLR, 2018, pp. 958--967.

\bibitem{boob2023stochastic}
D.~Boob, Q.~Deng, and G.~Lan, ``Stochastic first-order methods for convex and
  nonconvex functional constrained optimization,'' \emph{Mathematical
  Programming}, vol. 197, no.~1, pp. 215--279, 2023.

\bibitem{akhtar2021conservative}
Z.~Akhtar, A.~Singh~Bedi, and K.~Rajawat, ``Conservative stochastic
  optimization with expectation constraints,'' \emph{{IEEE} Trans. Signal
  Process.}, vol.~69, pp. 3190--3205, 2021.

\bibitem{thomdapu2021optimizing}
S.~T. Thomdapu and K.~Rajawat, ``Optimizing {QOS} for erasure-coded wireless
  data centers,'' in \emph{IEEE Intl. Conf. on Commun.}, 2021, pp. 1--6.

\bibitem{thomdapu2023stochastic}
S.~T. Thomdapu, H.~Vardhan, and K.~Rajawat, ``Stochastic compositional gradient
  descent under compositional constraints,'' \emph{{IEEE} Trans. Signal
  Process.}, vol.~71, pp. 1115--1127, 2023.

\bibitem{necoara2021minibatch}
I.~Necoara and A.~Nedi{\'c}, ``Minibatch stochastic subgradient-based
  projection algorithms for feasibility problems with convex inequalities,''
  \emph{Computational Optimization and Applications}, vol.~80, no.~1, pp.
  121--152, 2021.

\bibitem{wang2016stochastic}
M.~Wang and J.~Liu, ``A stochastic compositional gradient method using markov
  samples,'' in \emph{Proc. of the IEEE WSC}, Dec. 2016, pp. 702--713.

\bibitem{wei2018solving}
X.~Wei, H.~Yu, Q.~Ling, and M.~J. Neely, ``Solving non-smooth constrained
  programs with lower complexity than $\mathcal{O}(1/\varepsilon)$: a
  primal-dual homotopy smoothing approach,'' in \emph{Proc. of the Intl. Conf.
  on Neural Information Processing Systems}, 2018, pp. 3999--4009.

\bibitem{yang2017richer}
T.~Yang, Q.~Lin, and L.~Zhang, ``A richer theory of convex constrained
  optimization with reduced projections and improved rates,'' in
  \emph{International Conf. on Machine Learning}.\hskip 1em plus 0.5em minus
  0.4em\relax PMLR, 2017, pp. 3901--3910.

\bibitem{guo2022online}
H.~Guo, H.~Wei, X.~Liu, and L.~Ying, ``Online convex optimization with hard
  constraints: towards the best of two worlds and beyond,'' in \emph{Proc. of
  the Intl. Conf. on Neural Information Processing Systems}, 2022, pp.
  36\,426--36\,439.

\bibitem{sinha2024optimal}
A.~Sinha and R.~Vaze, ``Optimal algorithms for online convex optimization with
  adversarial constraints,'' in \emph{Proc. of the Intl. Conf. on Neural
  Information Processing Systems}, 2024, pp. 41\,274--41\,302.

\bibitem{wang2016stochastic2}
M.~Wang and D.~P. Bertsekas, ``Stochastic first-order methods with random
  constraint projection,'' \emph{SIAM Journal on Optimization}, vol.~26, no.~1,
  pp. 681--717, 2016.

\bibitem{idrees2025decentralized}
B.~M. Idrees, S.~D. Sharma, and K.~Rajawat, ``Decentralized stochastic
  successive convex approximation for composite non-convex problems with
  non-linear functional constraints,'' in \emph{IEEE ICASSP}, 2025.

\bibitem{bauschke1999strong}
H.~H. Bauschke, J.~M. Borwein, and W.~Li, ``Strong conical hull intersection
  property, bounded linear regularity, jameson’s property (g), and error
  bounds in convex optimization,'' \emph{Mathematical Programming}, vol.~86,
  no.~1, pp. 135--160, 1999.

\bibitem{duchi2018stochastic}
J.~C. Duchi and F.~Ruan, ``Stochastic methods for composite and weakly convex
  optimization problems,'' \emph{SIAM Journal on Optimization}, vol.~28, no.~4,
  pp. 3229--3259, 2018.

\bibitem{drusvyatskiy2019efficiency}
D.~Drusvyatskiy and C.~Paquette, ``Efficiency of minimizing compositions of
  convex functions and smooth maps,'' \emph{Mathematical Programming}, vol.
  178, no.~1, pp. 503--558, 2019.

\bibitem{khaled2023unified}
A.~Khaled, O.~Sebbouh, N.~Loizou, R.~M. Gower, and P.~Richt{\'a}rik, ``Unified
  analysis of stochastic gradient methods for composite convex and smooth
  optimization,'' \emph{Journal of Optimization Theory and Applications}, vol.
  199, no.~2, pp. 499--540, 2023.

\bibitem{gorbunov2020unified}
E.~Gorbunov, F.~Hanzely, and P.~Richt{\'a}rik, ``A unified theory of sgd:
  Variance reduction, sampling, quantization and coordinate descent,'' in
  \emph{International Conf. on Artificial Intelligence and Statistics}.\hskip
  1em plus 0.5em minus 0.4em\relax PMLR, 2020, pp. 680--690.

\bibitem{mishchenko2022proxskip}
K.~Mishchenko, G.~Malinovsky, S.~U. Stich, and P.~Richtarik, ``Prox{S}kip: Yes!
  local gradient steps provably lead to communication acceleration! finally!''
  in \emph{Proc. of the Intl. Conf. on Machine Learning}, 2022, pp.
  15\,750--15\,769.

\bibitem{allen2017katyusha}
Z.~Allen-Zhu, ``Katyusha: The first direct acceleration of stochastic gradient
  methods,'' \emph{The Journal of Machine Learning Research}, vol.~18, no.~1,
  pp. 8194--8244, 2017.

\bibitem{shang2018asvrg}
F.~Shang, L.~Jiao, K.~Zhou, J.~Cheng, Y.~Ren, and Y.~Jin, ``Asvrg: Accelerated
  proximal {SVRG},'' in \emph{Asian Conf. on Machine Learning}.\hskip 1em plus
  0.5em minus 0.4em\relax PMLR, 2018, pp. 815--830.

\bibitem{song2020variance}
C.~Song, Y.~Jinag, and Y.~Ma, ``Variance reduction via accelerated dual
  averaging for finite-sum optimization,'' in \emph{Proc. of the Intl. Conf. on
  Neural Information Processing Systems}, 2020, pp. 833--844.

\bibitem{xu2021iteration}
Y.~Xu, ``Iteration complexity of inexact augmented lagrangian methods for
  constrained convex programming,'' \emph{Mathematical Programming}, vol. 185,
  no.~1, pp. 199--244, 2021.

\bibitem{lin2018levelDet}
Q.~Lin, S.~Nadarajah, and N.~Soheili, ``A level-set method for convex
  optimization with a feasible solution path,'' \emph{SIAM Journal on
  Optimization}, vol.~28, no.~4, pp. 3290--3311, 2018.

\bibitem{yu2017simple}
H.~Yu and M.~J. Neely, ``A simple parallel algorithm with an $\mathcal{O}(1/t)$
  convergence rate for general convex programs,'' \emph{SIAM Journal on
  Optimization}, vol.~27, no.~2, pp. 759--783, 2017.

\bibitem{zermelo1931navigationsproblem}
E.~Zermelo, ``{\"U}ber das navigationsproblem bei ruhender oder
  ver{\"a}nderlicher windverteilung,'' \emph{ZAMM-Journal of Applied
  Mathematics and Mechanics/Zeitschrift f{\"u}r Angewandte Mathematik und
  Mechanik}, vol.~11, no.~2, pp. 114--124, 1931.

\bibitem{mercator}
\BIBentryALTinterwordspacing
{Mercator Ocean International}. (2025) Mercator ocean -- ocean forecasters.
  [Online; accessed 8-Sept-2025]. [Online]. Available:
  \url{https://www.mercator-ocean.eu/}
\BIBentrySTDinterwordspacing

\bibitem{copernicus_marine}
\BIBentryALTinterwordspacing
{Copernicus Marine Service}. (2025) Copernicus marine environment monitoring
  service. [Online; accessed 8-Sept-2025]. [Online]. Available:
  \url{https://marine.copernicus.eu/}
\BIBentrySTDinterwordspacing

\bibitem{jones2017planning}
D.~Jones and G.~A. Hollinger, ``Planning energy-efficient trajectories in
  strong disturbances,'' \emph{{IEEE} Robot. Autom. Lett.}, vol.~2, no.~4, pp.
  2080--2087, 2017.

\bibitem{song2019set}
H.~Song, P.~Shi, C.-C. Lim, W.-A. Zhang, and L.~Yu, ``Set-membership estimation
  for complex networks subject to linear and nonlinear bounded attacks,''
  \emph{IEEE {T}rans. on Neural Networks and Learning Systems}, vol.~31, no.~1,
  pp. 163--173, 2019.

\bibitem{year_prediction_msd_203}
T.~Bertin-Mahieux, ``{Year Prediction MSD},'' UCI Machine Learning Repository,
  2011, {DOI}: https://doi.org/10.24432/C50K61.

\end{thebibliography}

\newpage
\clearpage
\normalsize

\section*{Appendix D: Proof of Theorem \ref{avr-thm}}\label{pf-avr-thm}
For the sake of brevity, let us denote 
\begin{align}
	\delta^z_t &:= \EE\norm{\z_t - \x_\star}^2, & \tD_s &= \E{F(\tx_s)} - F(\x_\star).
\end{align}	
We begin with deriving some preliminary results. Using the definitions of $\x_t$, $\y_t$, $\z_t$, and $\z_{t-1}^+$ in Algorithm \ref{var-red-ssqp}, we see that 
\begin{align}\label{crossineq}
	\x_t-\y_t &= \alpha_s\br{\z_t-\z_{t-1}^+}, \\
	\y_t - \alpha_s\z_{t-1}^+ & = \br{1-\alpha_s-\omega_s}\x_{t-1} + \omega_s\tx_{s-1}. \label{yzrel}
\end{align}
The key to proving the required result is the following one-step inequality, which looks similar to the result in \cite[Lemma 6]{lan2019unified} but requires a different proof from that in the proximal case. 
\begin{lemma}\label{avr-lemma}
	If the parameters $\alpha_s$, $\omega_s$ and $\beta_s$ satisfy
	\begin{align} 
		\alpha_s + \omega_s &\leq 1 \label{avr-lemma-conda}\\
		1+\mu\beta_s-\alpha_s\beta_sL_\gamma &> 0 \label{avr-lemma-condb} \\
		\omega_s - \frac{\alpha_s\beta_sL_f}{1+\mu\beta_s-\alpha_s\beta_sL_\gamma} &\geq 0 ,\label{avr-lemma-condc}
	\end{align}
	then it holds that
	\begin{align}\label{rec}
		&\tfrac{\beta_s}{\alpha_s}\Delta_t + \br{1+\mu\beta_s}\tfrac{1}{2}\delta^z_t  \nonumber\\
		&\leq \tfrac{\beta_s}{\alpha_s}(1-\alpha_s-\omega_s)\Delta_{t-1}+ \tfrac{\beta_s\omega_s}{\alpha_s}\tD_{s-1}+ \tfrac{1}{2}\delta^z_{t-1}.
	\end{align}
\end{lemma}
\begin{IEEEproof}
	We begin with using the equality in \eqref{crossineq} and the $L_g$-smoothness of $g_k$ to obtain
	\begin{align}
		&g_k(\y_t) + \alpha_s\ip{\nabla g_k(\y_t)}{\z_t - \z_{t-1}^+} \nonumber\\
		&\eqfrom{crossineq}  g_k(\y_t) + \ip{\nabla g_k(\y_t)}{\x_t - \y_t} \gfrom{a1} g_k(\y_t) - \frac{L_g}{2}\norm{\x_t - \y_t}^2 \nonumber\\
		&\eqfrom{crossineq}  g_k(\y_t) - \frac{\alpha_s^2L_g}{2}\norm{\z_t - \z_{t-1}^+}^2. \label{avr-proof1}
	\end{align}	
	Further, from Jensen's inequality, and from \eqref{avr-zplsup}, we obtain
	\begin{align}
		\norm{\z_t-\z_{t-1}^+}^2 \leq \tfrac{1}{1+\mu\beta_s}\norm{\z_t-\z_{t-1}}^2 + \tfrac{\mu\beta_s}{1+\mu\beta_s}\norm{\z_t-\y_t}^2. \label{jensen}
	\end{align}
	Next, since $\z_t$-update in \eqref{avr-zup} involves minimizing an $\alpha_s(1+\mu\beta_s)$-strongly convex function, we have that
	\begin{align}
		&\tfrac{\alpha_s\beta_s\mu}{2}\norm{\y_t-\x_\star}^2+\tfrac{\alpha_s}{2}\norm{\z_{t-1}-\x_\star}^2 -\tfrac{\br{1+\mu\beta_s}\alpha_s}{2}\norm{\z_t-\x_\star}^2 \nonumber\\
		&+ \gamma \beta_s\max\{[g_k(\y_t) + \alpha_s\ip{\nabla g_k(\y_t)}{\x_\star - \z_{t-1}^+}]_+\} \nonumber\\
		&\gfrom{sclb} \alpha_s\beta_s\ip{\nt_t}{\z_t-\x_\star} +  \frac{\alpha_s}{2}\norm{\z_{t-1}-\z_t}^2 \nonumber\\
		&+ \gamma \beta_s\max\{[g_k(\y_t) + \alpha_s\ip{\nabla g_k(\y_t)}{\z_t - \z_{t-1}^+}]_+\}  \nonumber\\
		& +\frac{\alpha_s\beta_s\mu}{2}\norm{\y_t-\z_t}^2 + \alpha_s\beta_sh(\z_t) - \alpha_s\beta_sh(\x_\star)\\
		&\gfrom{avr-proof1}\alpha_s\beta_s\ip{\nt_t}{\z_t-\x_\star} +  \frac{\alpha_s}{2}\norm{\z_{t-1}-\z_t}^2 \nonumber\\
		&+ \gamma \beta_s\max\{[g_k(\x_t)]_+\} - \frac{\gamma\alpha_s^2\beta_sL_g}{2}\norm{\z_t - \z_{t-1}^+}^2 \nonumber\\
		& +\frac{\alpha_s\beta_s\mu}{2}\norm{\y_t-\z_t}^2 + \alpha_s\beta_sh(\z_t) - \alpha_s\beta_sh(\x_\star).\label{tpM} 
	\end{align}
	Substituting \eqref{jensen} into \eqref{tpM} and re-arranging, we obtain
	\begin{align}\label{maist}
		&\frac{\alpha_s\beta_s\mu}{2}\norm{\y_t-\x_\star}^2+\frac{\alpha_s}{2}\norm{\z_{t-1}-\x_\star}^2 \nonumber\\
		&+ \gamma \beta_s\max\{[g_k(\y_t) + \alpha_s\ip{\nabla g_k(\y_t)}{\x_\star - \z_{t-1}^+}]_+\} \nonumber\\
		&-\br{1+\mu\beta_s}\frac{\alpha_s}{2}\norm{\z_t-\x_\star}^2 - \alpha_s\beta_sh(\z_t) + \alpha_s\beta_sh(\x_\star) \nonumber\\
		&\geq \alpha_s\beta_s\ip{\nt_t}{\z_t-\x_\star} + \gamma \beta_s\max\{[g_k(\x_t)]_+\} \nonumber\\
		&+ \frac{\br{1+\mu\beta_s}\alpha_s-\alpha_s^2\beta_s\gamma L_g}{2}\norm{\z_t - \z_{t-1}^+}^2 .
	\end{align} 
	
	Further, from the convexity of $g_k$ and \eqref{avr-lemma-conda}, we have that
	\begin{align}
		g_k&(\y_t) + \alpha_s\ip{\nabla g_k(\y_t)}{\x_\star - \z_{t-1}} \nonumber\\
		&\leq g_k(\y_t + \alpha_s(\x_\star-\z_{t-1})) \nonumber\\
		&\eqfrom{yzrel} g_k((1-\alpha_s-\omega_s)\x_{t-1} + \alpha_s \x_\star + \omega_s\tx_{s-1}) \nonumber\\
		&\leq(1-\alpha_s-\omega_s)g_k(\x_{t-1}) + \alpha_sg_k(\x_\star) + \omega_sg_k(\tx_{s-1}) \nonumber \\
		&\leq (1-\alpha_s-\omega_s)g_k(\x_{t-1})  + \omega_sg_k(\tx_{s-1}) , \label{avr-proof2}
	\end{align}
	where the last inequality uses the fact that $g_k(\x_\star) \leq 0$. Therefore, from the monotonicity of the $\max{[\cdot]_+}$ operator, we obtain
	\begin{align}
		&\max\{[g_k(\y_t) + \alpha_s\ip{\nabla g_k(\y_t)}{\x_\star - \z_{t-1}}]_+\} \nonumber\\
		&\lfrom{avr-proof2}\max\{[(1-\alpha_s-\omega_s)g_k(\x_{t-1}) + \omega_sg_k(\tx_{s-1})]_+\}\nonumber\\
		&\lfrom{triangle2} (1-\alpha_s-\omega_s)\max\{[g_k(\x_{t-1})]_+\} \nonumber\\
		&\hspace{1cm}+\omega_s\max\{[g_k(\tx_{s-1})]_+\}. \label{avr-proof3}
	\end{align}
	Substituting \eqref{avr-proof3} into \eqref{maist}, we obtain
	\begin{align} \label{maincon}
		\alpha_s\beta_s&\ip{\nt_t}{\z_t-\x_\star} +  \frac{\br{1+\mu\beta_s}\alpha_s-\alpha_s^2\beta_s\gamma L_g}{2}\norm{\z_t - \z_{t-1}^+}^2 \nonumber\\
		&+ \gamma \beta_s\max\{[g_k(\x_t)]_+\}  + \br{1+\mu\beta_s}\frac{\alpha_s}{2}\norm{\z_t-\x_\star} ^2 \nonumber\\
		&+ \alpha_s\beta_sh(\z_t) - \alpha_s\beta_sh(\x_\star)\nonumber\\
		&\hspace{-5mm}\leq \frac{\alpha_s\beta_s\mu}{2}\norm{\y_t-\x_\star}^2 + \frac{\alpha_s}{2}\norm{\z_{t-1}-\x_\star}^2 \nonumber\\
		&+ \gamma \beta_s(1-\alpha_s-\omega_s)\max\{[g_k(\x_{t-1})]_+\} \nonumber \\
		& +\gamma\beta_s\omega_s\max\{[g_k(\tx_{s-1})]_+\}. 
	\end{align}
	
	Since $f$ is $L_f$-smooth and convex, we have that
	\begin{align}
		&f(\x_t) \lfrom{qub} f(\y_t) + \ip{\nabla f(\y_t)}{\x_t-\y_t} + \frac{L_f}{2}\norm{\x_t-\y_t}^2 \nonumber\\
		&\eqtext{\eqref{crossineq},\eqref{avr-xup}} f(\y_t) + (1-\alpha_s-\omega_s)\ip{\nabla f(\y_t)}{\x_{t-1}-\y_t} \nonumber\\
		&+ \alpha_s\ip{\nabla f(\y_t)}{\z_t-\y_t} + \omega_s \ip{\nabla f(\y_t)}{\tx_{s-1}-\y_t}  \nonumber\\
		&+ \frac{\alpha_s^2L_f}{2}\norm{\z_t-\z_{t-1}^+}^2\\
		&\lfrom{a1} (1-\alpha_s-\omega_s)f(\x_{t-1})  \label{avr-proof4}\\
		& + \alpha_s(f(\y_t) + \ip{\nabla f(\y_t)}{\x_\star - \y_t} + \ip{\nabla f(\y_t)}{\z_t - \x_\star}) \nonumber\\
		& + \omega_s (f(\y_t) + \ip{\nabla f(\y_t)}{\tx_{s-1}-\y_t}) + \tfrac{\alpha_s^2L_f}{2}\norm{\z_t-\z_{t-1}^+}^2. \nonumber
	\end{align}
	Adding \eqref{avr-proof4}  and \eqref{maincon}, we obtain
	\begin{align}\label{recint}
		f&(\x_t)  + \gamma \max\{[g_k(\x_t)]_+\} \leq -\alpha_s\ip{\nt_t}{\z_t-\x_\star} \nonumber\\
		&- \frac{\br{1+\mu\beta_s}\alpha_s-\beta_s\alpha_s^2(L_f+\gamma L_g)}{2\beta_s}\norm{\z_{t}-\z_{t-1}^+}^2 \nonumber\\
		& + \frac{\alpha_s}{2\beta_s}\norm{\z_{t-1}-\x_\star}^2 - \br{1+\mu\beta_s}\frac{\alpha_s}{2\beta_s}\norm{\z_t-\x_\star}^2 \nonumber\\
		&+\omega_s(f(\tx_{s-1}) + \gamma \max\{[g_k(\tx_{s-1})]_+\} )  \nonumber\\
		&+  \omega_s ((f(\y_t) + \ip{\nabla f(\y_t)}{\tx_{s-1}-\y_t}) - f(\tx_{s-1})) \nonumber\\
		&+ \alpha_s\Big(f(\y_t) + \ip{\nabla f(\y_t)}{\x_\star - \y_t} + \frac{\mu}{2}\norm{\y_t-\x_\star}^2\Big) \nonumber\\
		&+  (1-\alpha_s-\omega_s)\Big(f(\x_{t-1}) + \gamma \max\{[g_k(\x_{t-1})]_+\}\Big) \nonumber\\
		& +\alpha_s\ip{\nabla f(\y_t)}{\z_t -\x_\star}  - \alpha_sh(\z_t) + \alpha_sh(\x_\star) .
	\end{align}
	Moreover, by convexity of $h$, we have that
	\begin{align}\label{cvxH}
		-\alpha_sh(\z_t) \leq -h(\x_t) + (1-\alpha_s-\omega_s)h(\x_{t-1}) + \omega_sh(\tx_{s-1}).
	\end{align}
	Substituting \eqref{cvxH} in \eqref{recint} and using the facts that $f$ is $\mu$-strongly convex and $\max\{[g_k(\x_\star)]_+\} =0$, we obtain
	\begin{align}
		F(\x_t) &\leq \alpha_s\ip{\nabla f(\y_t)-\nt_t}{\z_t-\x_\star} \nonumber \\
		&- \frac{\br{1+\mu\beta_s}\alpha_s-\beta_s\alpha_s^2(L_f+\gamma L_g)}{2\beta_s}\norm{\z_{t}-\z_{t-1}^+}^2  \nonumber\\
		&- \br{1+\mu\beta_s}\frac{\alpha_s}{2\beta_s}\norm{\z_t-\x_\star}^2  +  (1-\alpha_s-\omega_s)F(\x_{t-1}) \nonumber\\
		&+ \alpha_s F(\x_\star) +\omega_sF(\tx_{s-1})  + \frac{\alpha_s}{2\beta_s}\norm{\z_{t-1}-\x_\star}^2\nonumber\\
		&-  \omega_s D_f(\tx_{s-1},\y_t). \label{one-step0}
	\end{align}
	Next, we take expectation with respect to $i_t$ and consider the different terms in \eqref{one-step0} separately. First note that since $\y_t$ and $\z_{t-1}^+$ are independent of $i_t$ and hence $\Et{\nt} = \nabla f(\y_t)$, the first term in \eqref{one-step0} can be bounded as
	\begin{align}
		&\alpha_s\Et{\ip{\nabla f(\y_t)-\nt_t}{\z_t-\x_\star}} \nonumber\\
		&= \alpha_s\Et{\ip{\nabla f(\y_t)-\nt_t}{\z_t-\z_{t-1}^+}} \nonumber\\
		& \hspace{1cm}+ \alpha_s\ip{\nabla f(\y_t)-\Et{\nt_t}}{\z_{t-1}^+-\x_\star} \nonumber\\
		& \lfrom{young} \frac{\alpha_s\beta_s}{2(1+\mu\beta_s-\alpha_s\beta_sL_\gamma)}\Et{\norm{\nabla f(\y_t)-\nt_t}^2} \nonumber\\
		&+ \frac{(1+\mu\beta_s)\alpha_s-\alpha_s^2\beta_sL_\gamma}{2\beta_s}\Et{\norm{\z_t -\z_{t-1}^+}^2},  \label{avr-proof5}
	\end{align}
	where recall that $L_\gamma := L_f + \gamma L_g$ and from \eqref{avr-lemma-condb}. Here, since $\Et{\nabla f_{i_t}(\tx_{s-1})} = \nabla f(\tx_{s-1})$, the variance of $\nt$ can be bounded as
	\begin{align}
		&\Et{\norm{\nt_t-\nabla f(\y_t)}^2} \leq \Et{\norm{\nabla f_{i_t}(\y_t)-\nabla f_{i_t}(\tx_{s-1})}^2} \nonumber\\
		&\lfrom{coco} 2L_f(\Et{f_{i_t}(\tx_{s-1}) - f_{i_t}(\y_t) - \ip{\nabla f_{i_t}(\y_t)}{\tx_{s-1}-\y_t}})\nonumber\\
		&= 2L_fD_f(\tx_{s-1},\y_t) .\label{avr-proof6}
	\end{align}
	Substituting \eqref{avr-proof6} into \eqref{avr-proof5}, and adding with \eqref{one-step0} after taking expectation with respect to $i_t$, we obtain
	\begin{align}
		&\Et{F(\x_t) + (1+\mu\beta_s)\frac{\alpha_s}{2\beta_s}\norm{\z_t-\x_\star}^2} \nonumber\\
		&\leq (1-\alpha_s-\omega_s)F(\x_{t-1}) + \frac{\alpha_s}{2\beta_s}\norm{\z_{t-1} - \x_\star}^2 \nonumber\\
		&+ \omega_s F(\tx_{s-1}) + \alpha_sF(\x_\star) \nonumber\\
		& 	-(\omega_s - \frac{\alpha_s\beta_sL_f}{1+\mu\beta_s-\alpha_s\beta_sL_\gamma})D_f(\tx_{s-1},\y_t),
	\end{align}
	where observe that the last term is non-positive from \eqref{avr-lemma-condc} and can be dropped. Taking full expectation and re-arranging, we obtain the required result in \eqref{rec}
\end{IEEEproof}

It is remarked that the conditions required in \eqref{avr-lemma-conda}-\eqref{avr-lemma-condc} are satisfied by the choice of parameters in the statement of Theorem \ref{avr-thm}. The statement of Lemma \ref{avr-lemma} will subsequently be used for each case in Theorem \ref{avr-thm}. 

\subsection{Proof of Theorem \ref{avr-thm}(1)}\label{pfncsupply}
For general convex function $f$, setting $\mu=0$ in \eqref{rec},
\begin{align}\nonumber
	&\tfrac{\beta_s}{\alpha_s}\Delta_t + \tfrac{1}{2}\delta^z_t \leq \tfrac{\beta_s}{\alpha_s}(1-\alpha_s-\omega_s)\Delta_{t-1} +\tfrac{\beta_s\omega_s}{\alpha_s} \tD_{s-1} + \tfrac{1}{2}\delta^z_{t-1}.
\end{align}
Substituting $\theta_t$ as specified in \eqref{thetacond} at line 10 of Alg. \ref{var-red-ssqp} and summing the recursive expression for $t = 1$, $\ldots$, $T_s$, 
\begin{align}
	\tD_s\sum_{t=1}^{T_s} \theta_t \leq \sum_{t=1}^{T_s} \theta_t\Delta_t  \leq &\bsq{\tfrac{\beta_s}{\alpha_s}(1-\alpha_s) + (T_s-1)\tfrac{\beta_s\omega_s}{\alpha_s}}\tD_{s-1} \nonumber\\
	&+ \tfrac{1}{2}(\delta^z_{T_{s-1}}-\delta^z_{T_s}),
\end{align}
where, we have utilized the fact that $\x_0 = \tx_{s-1}$,  $\tx_s = \sum_t \theta_t\x_t /\sum_t\theta_t$, and the convexity of $F$. Let us denote 
\begin{align}
	\cL_s &:\eqtext{\eqref{thetacond}}\sum_{t=1}^{T_s} \theta_t  = \frac{\beta_s}{\alpha_s} + (T_s-1)\frac{\beta_s\br{\alpha_s+\omega_s}}{\alpha_s}, \\
	\cR_s &:= \frac{\beta_s}{\alpha_s}(1-\alpha_s) + (T_s-1)\frac{\beta_s\omega_s}{\alpha_s},
\end{align}
so that $\cL_s\tD_s = \cR_s\tD_{s-1} + \tfrac{1}{2}(\delta^z_{T_{s-1}}-\delta^z_{T_s})$. Summing over $j = 1,..,s$  and rewriting, we obtain $\cL_s\Delta_s+ \sum_{j=1}^{s-1}\br{\cL_j - \cR_{j+1}}\tD_j \leq \cR_1\tD_0 + \tfrac{1}{2}\delta^z_0-\tfrac{1}{2}\delta^z_{T_s}$ where the last term can be dropped. 

Next, we can show that $\nu_s := \cL_s - \cR_{s+1} \geq 0$ for $s\geq 1$. For $1\leq s<s_0$, we have $\alpha_s = \alpha_{s+1} = \omega_s = 1/2$, $\beta_s = \beta_{s+1}$, and $T_{s+1} = 2T_s$, so that $\nu_s = 0$. For $s\geq s_0$,
\begin{align}
	\nu_s  =& \tfrac{\beta_s}{\alpha_s} - \tfrac{\beta_{s+1}}{\alpha_{s+1}} + \beta_{s+1} + \br{T_{s_0} -1}\bsq{\tfrac{\beta_s\br{\alpha_s+\omega_s}}{\alpha_s} - \tfrac{\beta_{s+1}\omega_{s+1}}{\alpha_{s+1}}} \nonumber\\
	=& \tfrac{1}{24L_\gamma}\br{2+ \br{T_{s_0} - 1}(2(s-s_0+4)-1)} \geq 0.
\end{align}
Setting $\bar{\x}_s := \sum_{j=1}^{s-1} \nu_j\tx_j/\sum_{k=1}^{s-1} \nu_k$ and using the convexity of $F$, we obtain
\begin{align}\label{intrbp}
	\cL_s\tD_s + \br{\E{F(\bar{\x}_s)} - F(\x_\star)} &\sum_{j=1}^{s-1}\nu_j \leq \cR_1\tD_0 + \tfrac{1}{2}\delta^z_{T_0}\nonumber\\
	\Rightarrow \cL_s\tD_s  &\leq \cR_1\tD_0 + \tfrac{1}{2}\delta^z_{T_0},
\end{align}
where from \eqref{penalty1}, we have used $F(\bar{\x}_s) \geq F(\x_\star)$. Here, $\cR_1= 2/3L_\gamma$. In the case when $1\leq s\leq s_0$, we have that $\cL_s = \frac{2^{s+1}}{3L_\gamma}$ so that $\tD_s \leq 2^{-\br{s+1}}D_0$. For $s>s_0$, we have
\begin{align}
	\cL_s &= \tfrac{1}{3L_\gamma\alpha_s^2}\bsq{1+\br{T_{s_0}-1}\br{\alpha_s+\tfrac{1}{2}}} \nonumber\\
	&= \tfrac{\br{s-s_0+4}\br{T_{s_0}-1}}{6L_\gamma} + \tfrac{\br{s-s_0+4}^2\br{T_{s_0}+1}}{24L_\gamma}\geq \tfrac{\br{s-s_0+4}^2n}{48L_\gamma},\nonumber
\end{align}
where the last inequality follows from $T_{s_0} = 2^{\lfloor\log_2n\rfloor} \geq n/2$. Hence we obtain
\begin{align} \label{smesc}
	\tD_s \leq \tfrac{16D_0}{\br{s-s_0+4}^2n}.
\end{align}
First consider the case when $n\geq D_0/\epsilon$ for a given $\epsilon$. The algorithm cannot run for more than $s_0$ epochs, which can easily be checked as $2^{-\br{S_l-1}}D_0 \leq \eps$ which implies, that the total number of epochs the algorithm runs is given by 
\begin{align}
	S_l = \min\bc{\log{\tfrac{D_0}{\eps}}, s_0} = \log{\tfrac{D_0}{\eps}},\label{lowac2}
\end{align}
since $s_0 =  \lfloor\log n\rfloor + 1 \geq \log \tfrac{D_0}{\eps}$. The SFO and QMO complexities are then given by 
\begin{align}
	N_{\text{QMO}} &= \sum_{s=1}^{S_l}T_s = \cO\br{\min \bc{\tfrac{D_0}{\eps},n}}=\cO\br{\tfrac{D_0}{\eps}}, \label{case1q}\\
	N_{\text{SFO}} &= nS_l + \sum_{s=1}^{S_l}T_s = \cO\br{n\log{\tfrac{D_0}{\eps}}}, \label{case1s}
\end{align}
where we have used the fact that $n \geq \tfrac{D_0}{\epsilon}$. Next we consider the case when $n < D_0/\epsilon$ and it is possible to evaluate true gradient for more than $s_0$ epochs as
\begin{align}
	S_h = \left\lceil \sqrt{\tfrac{16D_0}{n\epsilon}} + s_0 -4 \right\rceil.
\end{align} 
The total number of gradient evaluations of $f_i$ is
\begin{align}
	\label{case2s}
	N_{\text{SFO}} &= ns_0 + \sum_{s=1}^{s_0}T_s + \br{S_h - s_0}\br{n + T_{s_0}} \\
	& = ns_0 + 2^{s_0}-1 + \br{n + T_{s_0}}S_h - s_02^{s_0-1} -ns_0 \nonumber\\
	&\leq \br{n + T_{s_0}}S_h \leq \br{n + 2^{s_0-1}}\br{\sqrt{\tfrac{16D_0}{n\epsilon}} + s_0}\nonumber\\
	&\leq \br{n + n}\br{\sqrt{\tfrac{16D_0}{n\epsilon}} + \log{n}} = \cO\br{n\log{n} + \sqrt{\tfrac{nD_0}{\epsilon}} },\nonumber\\
\end{align}
and,
\begin{align}
	N_{\text{QMO}} &= \sum_{s=1}^{s_0}T_s+T_{s_0}\br{S_h - s_0}\label{case2q}\\
	=&(2^{s_0}-1)+(2^{s_0-1})\br{\sqrt{\tfrac{16nD_0}{\eps}}-4}=\cO\br{\sqrt{\tfrac{nD_0}{\eps}}}.\nonumber
\end{align}
Combining, we obtain the desired bound in \eqref{avr-c-res}. 

\subsection{Proof of Theorem \ref{avr-thm}(2)}\label{pfscsupply}
For this result, we separately consider different cases based on values of $s$ and $n$. 
\subsubsection{Case $s\leq s_0$} In this case $\alpha_s = \omega_s = \frac{1}{2}$, $\beta_s = \frac{2}{3L_\gamma}$ and $T_s = 2^{s-1}$, so we can write \eqref{rec} as
\begin{align}
	&\frac{\beta_s}{\alpha_s}\Delta_t + \br{1+\mu\beta_s}\tfrac{1}{2}\delta^z_t \leq  \frac{\beta_s\omega_s}{\alpha_s} \tD_{s-1} + \tfrac{1}{2}\delta^z_{t-1}.
\end{align} 
Summing over $t=1$, $\ldots$, $T_s$, we obtain
\begin{align}
	\tfrac{\beta_s}{\alpha_s}&\sum_{t=1}^{T_s} \Delta_t + \tfrac{1}{2}\delta^z_{T_s} + \tfrac{\mu\beta_s}{2}\sum_{t=1}^{T_s}\delta^z_t \leq \tfrac{\beta_sT_s}{2\alpha_s} \tD_{s-1} + \tfrac{1}{2}\delta^z_0.
\end{align}
Using the definitions of $\tx_s$ and $\theta_t$, we obtain
\begin{align}
	\tfrac{4T_s}{3L_\gamma}\tD_s + \tfrac{1}{2}\delta^z_{T_s} &\leq \tfrac{4T_s}{6L_\gamma}\tD_{s-1} + \tfrac{1}{2}\delta^z_{T_{s-1}} \nonumber\\
	&= \tfrac{4T_{s-1}}{3L_\gamma}\tD_{s-1} + \tfrac{1}{2}\delta^z_{T_{s-1}}.
\end{align}
Applying the inequality recursively over $s$, we get
\begin{align} \label{nexL}
	\tfrac{4T_s}{3L_\gamma}&\tD_s + \tfrac{1}{2}\delta^z_{T_s} \leq \tfrac{4}{3L_\gamma}\tD_0 + \tfrac{1}{2}\delta^z_{T_0}.
\end{align}
By substituting $T_s = 2^{s-1}$, we conclude $\tD_s \leq 2^{-(s+1)}D_0$. Now by \eqref{case1s} and \eqref{case1q} we get the SFO and QMO complexities.

\subsubsection{Case  $s>s_0$ and $n \geq \frac{3\kappa}{4}$} In this case $\alpha_s = \omega_s = \frac{1}{2}$, $\beta_s = \frac{2}{3L_\gamma}$, and $T_s = T_{s_0} = 2^{s_0-1}$ so \eqref{rec} yields
\begin{align}
	\tfrac{4}{3L_\gamma}&\Delta_t + \br{1+\tfrac{2}{3\kappa}}\tfrac{1}{2}\delta^z_t \leq  \tfrac{2}{3L_\gamma} \tD_{s-1} + \tfrac{1}{2}\delta^z_{t-1}.
\end{align} 
Multiplying both sides by $\theta_t = \Gamma_{t-1} = \br{1+\tfrac{2}{3\kappa}}^{t-1}$,
\begin{align}
	&\frac{4\Gamma_{t-1}}{3L_\gamma}\Delta_t + \frac{\Gamma_{t}}{2}\delta^z_t \leq  \frac{2\Gamma_{t-1}}{3L_\gamma} \tD_{s-1} + \frac{\Gamma_{t-1}}{2}\delta^z_{t-1}.
\end{align}
Summing over $t=1$, $\ldots$, $T_s$, we obtain
\begin{align}\nonumber
	\tfrac{4}{3L_\gamma}\sum_{t=1}^{T_s}\theta_t&\Delta_t + \frac{\Gamma_{T_s}}{2}\delta^z_{T_s} \leq  \frac{2}{3L_\gamma}\tD_{s-1} \sum_{t=1}^{T_s}\theta_t + \frac{1}{2}\delta^z_{T_{s-1}}.
\end{align}
For $s\geq s_0$, it holds that $n\geq 2^{s_0-1} = T_{s_0} = 2^{\lfloor\log n\rfloor} \geq n/2$, and hence it follows
\begin{align}
	\Gamma_{T_s} = \br{1+\tfrac{2}{3\kappa}}^{T_s} &\geq 1+\tfrac{2\mu T_{s_0}}{3L_\gamma}\geq 1+\tfrac{T_{s_0}}{2n} &\geq \tfrac{5}{4}.
\end{align}
Denoting $\Theta_s := \sum_{t=1}^{T_s}\theta_t \geq T_s$ and applying these inequalities we obtain
\begin{align}
	&\tfrac{5}{4}\times\tfrac{2}{3L_\gamma}\tD_s +\tfrac{5}{4\Theta_s}\times\tfrac{1}{2}\delta^z_{T_s}  \leq  \tfrac{2}{3L_\gamma}\tD_{s-1} +\tfrac{1}{2\Theta_s}\delta^z_{T_{s-1}},
\end{align}
which upon continuing recursively for $s \geq s_0$, yields
\begin{align}
	\tfrac{2}{3L_\gamma}\tD_s + \tfrac{1}{2\Theta_s}\delta^z_{T_s} &\leq \br{\tfrac{4}{5}}^{s-s_0} \left[ \tfrac{2}{3L_\gamma}\tD_{s_0}+\tfrac{1}{2T_{s_0}}\delta^z_{T_{s_0}}\right] \nonumber\\
	&\hspace{-1cm}\leqtext{\eqref{nexL}}\br{\tfrac{4}{5}}^{s-s_0}\left[\tfrac{4}{3L_\gamma T_{s_0}}\tD_0 + \tfrac{1}{2T_{s_0}}\delta^z_{T_0}\right].
\end{align}
Substituting $T_{s_0} = 2^{s_0-1}$ and $D_0$, we obtain the final result for this case as $\tD_s \leq \br{\frac{4}{5}}^{s-s_0} \frac{D_0}{2^{s_0}} \leq \br{\frac{4}{5}}^s D_0$. Observing that VARAS runs for $S = \cO(\log\frac{D_0}{\eps})$ epochs, we bound SFO and QMO evaluations as,
\begin{align*}
	N_{\text{SFO}}&=nS+\sum_{s=1}^{S}T_s \leq 2nS = \cO\br{n\log\tfrac{D_0}{\eps}}, \\
	N_{\text{QMO}}&=\sum_{s=1}^{S}T_s=\cO\br{n+(S-s_0)n} = \cO\br{n\log\tfrac{D_0}{\eps}}.
\end{align*}

\subsubsection{Case $s_0<s\leq s_0+\sqrt{\frac{12\kappa}{n}} -4$ and $n < \frac{3\kappa}{4}$}  In this case $\alpha_s = \frac{2}{s-s_0+4}$, $\omega_s = \frac{1}{2}$, $\beta_s = \frac{s-s_0+4}{6L_\gamma}$, and $T_s = T_{s_0} = 2^{s_0-1}$. Observe that the parameter setting is same as in the smooth convex case with $\mu =0$. Hence the same result holds for positive $\mu$ values too which is 
\begin{align}\label{sgrts0}
	\cL_s&\tD_s + \tfrac{1}{2}\delta^z_{T_s} \leq \cR_{s_0+1} \tD_{s_0} + \tfrac{1}{2}\delta^z_{T_{s_0}} \leq \tfrac{D_0}{3L_\gamma}.
\end{align}
where the last inequality follows since $\cL_{s_0} \geq \tfrac{2T_{s_0}}{3L_\gamma}$. From the analysis in previous subsection, $\cL_s \geq (s-s_0+4)^2\frac{n}{48L_\gamma}$ and hence $\tD_s \leq  \tfrac{16D_0}{\br{s-s_0+4}^2n}$. Following \eqref{case2s} and \eqref{case2q}, we bound gradient and QP evaluations.

\subsubsection{Case $s > s_0+\sqrt{\frac{12\kappa}{n}} - 4$ and $n < \frac{3\kappa}{4}$} In this case, $\alpha_s = \sqrt{\frac{n}{3\kappa}}$, $\omega_s = \tfrac{1}{2}$, $\beta_s = \frac{1}{\sqrt{3nL_\gamma \mu}}$, and $T_s = T_{s_0} = 2^{s_0-1}$. By multiplying with $\Gamma_{t-1}$ on both sides of \eqref{rec}, we obtain
\begin{align}
	\tfrac{\beta_s}{\alpha_s}&\Gamma_{t-1}\Delta_t + \Gamma_{t}\tfrac{1}{2}\delta^z_t \leq \frac{\beta_s\omega_s}{\alpha_s} \Gamma_{t-1}\tD_{s-1} \nonumber\\ 
	&+\tfrac{\beta_s}{\alpha_s}(1-\alpha_s-\omega_s)\Gamma_{t-1}\Delta_{t-1} + \Gamma_{t-1}\tfrac{1}{2}\delta^z_{t-1}.
\end{align} 
Summing over $t=1$, $\ldots$, $T_s$, we obtain
\begin{align}
	&\tfrac{\beta_s}{\alpha_s}\sum_{t=1}^{T_s}\theta_t\Delta_t  +\tfrac{\beta_s}{\alpha_s}\sum_{t=1}^{T_s-1}(1-\alpha_s-\omega_s)\Gamma_{t}\Delta_t \nonumber\\
	&\leq \tfrac{\beta_s}{\alpha_s}\sum_{t=1}^{T_s}(1-\alpha_s-\omega_s)\Gamma_{t-1}\Delta_{t-1} + \tfrac{\beta_s\omega_s}{\alpha_s} \sum_{t=1}^{T_s}\Gamma_{t-1}\tD_{s-1} \nonumber \\
	&\hspace{5mm}+\tfrac{1}{2}\delta^z_{T_{s-1}} - \tfrac{\Gamma_{T_s}}{2}\delta^z_{T_s}.
\end{align}
By canceling the terms, and simplifying, we get
\begin{align}\label{rewrt}
	&\tfrac{\beta_s}{\alpha_s}\br{\sum_{t=1}^{T_{s_0}}\theta_{t}}\tD_s + \tfrac{\Gamma_{T_{s_0}}}{2}\delta^z_{T_s} \nonumber\\
	&\leq \tfrac{\beta_s}{\alpha_s}\big[1-\alpha_s-\omega_s + \omega_s\sum_{t=1}^{T_{s_0}}\Gamma_{t-1}\big]\tD_{s-1}+ \tfrac{1}{2}\delta^z_{T_{s-1}}.
\end{align}
As in \cite[Lemma 11]{lan2019unified}, 
we can establish that $\sum_{t=1}^{T_{s_0}}\theta_t \geq \Gamma_{T_{s_0}}\Omega_{s_0}$ where
\begin{align}
	\Omega_{s_0} := \tfrac{\beta_s}{\alpha_s}&\Big[1-\alpha_s-\omega_s + \omega_s\sum_{t=1}^{T_{s_0}}\Gamma_{t-1}\Big]\geq \tfrac{(\bar{s}_0 - s_0 +4)^2T_{s_0}}{24L_\gamma},\nonumber
\end{align}
where $\bar{s}_0 = s_0+\sqrt{\frac{12\kappa}{n}} - 4$, so that \eqref{rewrt} can be written as
\begin{align}
	\Omega_{s_0}\tD_s+ \tfrac{1}{2}\delta^z_{T_s} &\leq \tfrac{1}{\Gamma_{T_{s_0}}}\left[\Omega_{s_0}\tD_{s-1} + \frac{1}{2}\delta^z_{T_{s-1}}\right] \nonumber\\
	&\leq \tfrac{1}{\Gamma_{T_{s_0}}^{s-\bar{s}_0}}\left[\Omega_{s_0}\tD_{\bar{s}_0} + \frac{1}{2}\delta^z_{T_{\bar{s}_0}}\right].
\end{align}

From the analysis in previous subsection, it holds $\cL_{\bar{s}_0} \geq  \frac{(\bar{s}_0-s_0+4)^2T_{s_0}}{24L_\gamma} = \frac{T_{s_0}}{2n\mu}$. Denoting $\cC_s =  (1+\mu\beta_s)^{-T_{s_0}(s-\bar{s}_0)}$, we conclude this case with
\begin{align}
	&\tD_s \leq \cC_s \tD_{\bar{s}_0}+ \tfrac{12L_\gamma \cC_s}{(\bar{s}_0 - s_0 +4)^2T_{s_0}}\delta^z_{T_{\bar{s}_0}} \nonumber\\
	&\leq  \tfrac{24L_\gamma \cC_s}{(\bar{s}_0 - s_0 +4)^2T_{s_0}} \br{\cL_{\bar{s}_0}\tD_{\bar{s}_0}+\tfrac{1}{2}\delta^z_{T_{\bar{s}_0}}}\nonumber\\
	&\leq  \tfrac{24L_\gamma \cC_s}{(\bar{s}_0 - s_0 +4)^2T_{s_0}} \tfrac{D_0}{3L_\gamma} \leq \br{1+\sqrt{\tfrac{}{3n\kappa}}}^{\tfrac{-n(s-\bar{s}_0)}{2}} \tfrac{D_0}{3\kappa/4}.
\end{align}
We note that, here the number of epochs is bounded by $S = \bar{s}_0+2\sqrt{\frac{3 L_\gamma}{n\mu}}\log\frac{4D_0}{3\kappa\eps}$. Thus the number of gradient evaluations is bounded by
\begin{align}
	\label{case4s}
	N_{\text{SFO}} =& \sum_{s=1}^{S}(n+T_s)\nonumber\\
	=&\sum_{s=1}^{s_0}(n+T_s)+\sum_{s=s_0+1}^{\bar{s}_0}(n+T_{s_0}) + (n+T_{s_0})(S-\bar{s}_0)\nonumber\\
	\leq&2n\log n+2n(\sqrt{\frac{12\kappa}{n}}-4)+4n\sqrt{\frac{12\kappa}{n}}\log\frac{4D_0}{3\kappa\eps}\nonumber\\
	=&\cO\br{n\log{n} + \sqrt{n\kappa}\log{\frac{4D_0}{3\kappa\eps}}},
\end{align}
and QP evaluations are bounded by 
\begin{align}
	\label{case4q}
	N_{\text{QMO}} =& \sum_{s=1}^{S}T_s =\sum_{s=1}^{s_0}T_s+\sum_{s=s_0+1}^{\bar{s}_0}T_{s_0} + T_{s_0}(S-\bar{s}_0)\nonumber\\
	=&\cO \br{n+n(\sqrt{\frac{12\kappa}{n}}-4)+2n\sqrt{\frac{12\kappa}{n}}\log\frac{4D_0}{3\kappa\eps}}\nonumber\\
	=&\cO\br{n\log{n} + \sqrt{n\kappa}\log{\frac{4D_0}{3\kappa\eps}}}.
\end{align}

	Now by making similar arguments as in previous subsection Sec. \ref{pfncsupply} and \cite[Thm. 2]{lan2019unified} we obtain a bound on $\tD_s$. Finally, we can follow the steps in \eqref{con_ineq_1}-\eqref{con_ineq_2} to obtain the desired bounds on the optimality gap and constraint violation. 

To find $\epsilon$-optimal solution, the number of gradient evaluations needed by the algorithm are,
\begin{align}
	N = \left\{\begin{matrix}
			\cO\br{n\log{\frac{D_0}{\eps}}}& n \geq \frac{D_0}{\epsilon} \text{ or } n \geq \frac{3\kappa}{4}\\ 
			\cO\br{n\log{n} + \sqrt{\frac{nD_0}{\epsilon}} } & n < \frac{D_0}{\epsilon} \leq \frac{3\kappa}{4} \\
			\cO\br{n\log{n} + n\kappa\log{\frac{D_0/\epsilon}{3L_\gamma/4\mu}}} & n < \frac{3\kappa}{4} \leq \frac{D_0}{\epsilon}.
			\end{matrix}\right.
\end{align}

\end{document}